\documentclass{aims}
\usepackage{amsmath}
\usepackage{amssymb}
\usepackage{paralist}
\usepackage{graphics} %% add this and next lines if pictures should be in esp format
\usepackage{epsfig} %For pictures: screened artwork should be set up with an 85 or 100 line screen
\usepackage{graphicx}  \usepackage{epstopdf} %This is to transfer .eps figure to .pdf figure; please compile your paper using PDFLaTex or PDFTeXify.
\usepackage[colorlinks=true]{hyperref}
\hypersetup{urlcolor=blue, citecolor=red}

\def\publname{}
\def\volinfo{}
\def\pageinfo{}

\usepackage[T1]{fontenc}
\usepackage[utf8]{inputenc}
\usepackage{lmodern}

\usepackage[table,dvipsnames]{xcolor}
\usepackage{hhline}
\usepackage{longtable}

\usepackage{algorithmic}
\usepackage[ruled,vlined]{algorithm2e}
\usepackage{adjustbox}
\usepackage{array}
\usepackage[caption=false,font=footnotesize]{subfig}
\usepackage{url}
\usepackage{enumitem}
\usepackage{multicol}
\usepackage{caption}
\usepackage[bb=boondox]{mathalfa}

\newtheorem{theorem}{Theorem}[section]
\newtheorem{corollary}{Corollary}

\newtheorem{lemma}[theorem]{Lemma}
\newtheorem{proposition}{Proposition}

\theoremstyle{definition}

\newtheorem{remark}{Remark}
\newtheorem{example}{Example}

\usepackage{pgf}
\usepackage{tikz}
\usetikzlibrary{calc,fadings,decorations.pathreplacing}
\usetikzlibrary{matrix,arrows}
\usepackage{collcell}
\usepackage{xstring}

\def\minus{%
  \setbox0=\hbox{-}%
  \vcenter{%
    \hrule width\wd0 height \the\fontdimen8\textfont3%
  }%
}

\newcommand{\node}[1]{\mathrm{#1}}

\newcommand{\tr}{\operatorname{tr}}

\newcommand{\ds}{\displaystyle}
\newcommand{\ts}{\textstyle}

\newcommand{\Nn}{{\mathbb N}}

\newcommand{\Rr}{{\mathbb R}}

\newcommand{\Ll}{\mathbf{L}}
\newcommand{\Aa}{\mathbf{A}}
\newcommand{\Bb}{\mathbf{B}}

\newcommand{\Kk}{\mathbf{K}}
\newcommand{\Ee}{\mathrm{E}}

\newcommand{\krycl}{\mathcal{K}^{\mathrm{cbl}}}
\newcommand{\krygl}{\mathcal{K}^{\mathrm{gbl}}}
\newcommand{\krysl}{\mathcal{K}^{\mathrm{sbl}}}
\newcommand{\krykr}{\mathcal{K}^{\mathrm{kr}}}

\title[Krylov subspace methods to accelerate kernel machines on graphs] % Running head is the full title or shortened version of the full title. This will appear at the top of odd pages. Please make sure it fits within the width limit.
      {Krylov subspace methods to accelerate \\ kernel machines on graphs} % Only the first word and proper nouns should be capitalized

\author[Wolfgang Erb]{}


 \email{wolfgang.erb@unipd.it}

\begin{document}
\maketitle

% Enter the first author's name and address:
\centerline{\scshape Wolfgang Erb}
\medskip
{\footnotesize
% Enter the address of the first author
 \centerline{Universit{\`a} degli Studi di Padova}
   \centerline{Dipartimento di Matematica ''Tullio Levi-Civita''}
   \centerline{Via Trieste 63, 35121 Padova, Italy}
   \centerline{wolfgang.erb@unipd.it} \vspace{2mm}
   \centerline{16/01/2023}
} % Do not forget to end {\footnotesize with the sign }

\bigskip

% The name of the associate editor will be entered by AIMS editorial staff.
% "Communicated by the associate editor name" is not needed for special issue.
 %\centerline{(Communicated by the associate editor name)}

%The abstract of your paper
\begin{abstract}
In classical frameworks as the Euclidean space, positive definite kernels as well as their analytic properties are explicitly available and can be incorporated directly in kernel-based learning algorithms.  This is different if the underlying domain is a discrete irregular graph. In this case, respective kernels have to be computed in a preliminary step in order to apply them inside a kernel machine. Typically, such a kernel is given as a matrix function of the graph Laplacian. Its direct calculation leads to a high computational burden if the size of the graph is very large. In this work, we investigate five different block Krylov subspace methods to obtain cheaper iterative approximations of these kernels. We will investigate convergence properties of these Krylov subspace methods and study to what extent these methods are able to preserve the symmetry and positive definiteness of the original kernels they are approximating. We will further discuss the computational complexity and the memory requirements of these methods, as well as possible implications for the kernel predictors in machine learning.  
\end{abstract}

\section{Introduction}

Learning algorithms based on positive definite kernels are among the most robust instruments in machine learning for classification and regression tasks. Some of the principal advantages of kernel machines compared to other popular learning tools as artificial neural networks are the existence of a well-founded statistical and mathematical theory of learning, usually relying on a reproducing kernel Hilbert space, the simple adaptability of kernel methods to different domains, as well as a simple implementability of the method \cite{Schoelkopf2002,Vapnik1998}. 

In this work, we review and investigate iterative algorithms to implement kernel machines efficently on graph domains. While in Euclidean spaces most relevant kernels can be formulated explicitly, kernels that describe diffusion processes or characterize smoothness spaces on irregular graph domains have to be computed numerically. In many cases, these kernels can be characterized and calculated as matrix functions $\phi(\Ll)$ of a graph Laplacian $\Ll$. Important examples are the diffusion kernel that can be described as an exponential function of $\Ll$ \cite{KondorLafferty2002}, and the variational spline kernel which can be written as inverse power of the shifted matrix $\Ll$ \cite{Pesenson2009,Ward2018interpolating}. The calculations of these matrix functions can get cost-intensive or even prohibitive if the size of the graph gets large and when the spectral decomposition of the graph Laplacian is required. 

To avoid the spectral decomposition of $\Ll$, a well-established strategy consists in approximating the matrix function $\phi(\Ll)$ with a matrix polynomial $p_{\phi,m}(\Ll)$ of degree $m$. In typical real world graphs, every graph node has only a few neighboring nodes, and the respective graph Laplacian $\Ll$ has a sparse structure. In this case, the polynomial approximation $p_{\phi,m}(\Ll)$ of $\phi(\Ll)$ gets particularly advantageous as only a few matrix-matrix products are necessary to obtain the approximate kernel matrix $p_{\phi,m}(\Ll)$. Furthermore, if just the action of the matrix $p_{\phi,m}(\Ll)$ to a single vector $x$ is required, it suffices essentially to calculate $m$ matrix-vector products in order to obtain $p_{\phi,m}(\Ll)x$. Such polynomial approximations $p_{\phi,m}(\Ll) x$ are inherently related to Krylov subspace methods. These iterative methods have been originally developed for the approximate resolution of large linear systems of equations and eigenvalue problems and have been intensively studied since the early works of Krylov, Gantmacher, Lanczos, and Hestenes \& Stiefel, see \cite{LiesenStrakos2012,Saad2003} for a general reference. For the approximation of matrix functions, in particular of the matrix exponential, a numerical analysis of the error and the stability of Arnoldi and Lanczos iterations has been first given in \cite{Gallopoulos1992,Saad1992} and was later refined in \cite{Hochbruck1997,Musco2018,stewartleyk1996}.

For the usage in kernel machines, the matrix polynomial $p_{\phi,m}(\Ll)$ has to satisfy additional properties. In typical supervised classification or regression problems labeled data is available at a set $W$ of $N$ nodes. Based on this labeled data a kernel predictor is calculated using $N$ columns $\phi(\Ll) \Ee_W$ of the kernel matrix $\phi(\Ll)$, where $\Ee_W$ is a block of $N$ unit vectors encoding the $N$ sampling nodes in $W$. For the iterative calculation of an approximate kernel predictor, we therefore require a polynomial block $p_{\phi,m}(\Ll) \Ee_W$ that approximates the columns of $\phi(\Ll) \Ee_W$. While it is possible to use ordinary Krylov subspace methods sequentially for each column in $\Ee_W$ to get an approximation of the entire block $\phi(\Ll) \Ee_W$, it has advantages to use methods that are specifically designed for blocks. In the literature, the respective iterative methods are referred to as block Krylov subspace methods for matrix functions, see \cite{Frommer2017,Gutknecht2006,Lund2018,Schmelzer2004} for a general introduction and \cite{elsworthguettel2020,LopezSimoncini2006,simoncini1995,simoncini1996} for more specific studies. These methods generalize the ordinary Krylov subspace methods and approximate matrix functions using the information of the entire initial block $\Ee_W$ instead of a single initial vector $x$ only. 

In Section \ref{sec:blockKrylov} of this work, we will introduce and investigate five different block Krylov subspace methods for the iterative approximation of the columns $\phi(\Ll) \Ee_W$. Three of these block methods are Lanczos-type algorithms in which the polynomial $p_{\phi,m}(\Ll)$ depends on the initial block $\Ee_W$, the two other methods are Chebyshev approximations in which the matrix polynomial $p_{\phi,m}(\Ll)$ is independent of the block $\Ee_W$ and linked to a polynomial that interpolates the function $\phi$ on a specific Chebyshev grid. 

To have highly efficient iterative schemes for large graphs, it is important that the block vector $p_{\phi,m}(\Ll) \Ee_W$ approximates $\phi(\Ll) \Ee_W$ already for small degrees $m$ such that only a few matrix-vector products are required in the calculations. For this, we will give explicit error estimates in Section \ref{sec:errorestimates} and Section \ref{sec:errorcbl} that guarantee convergence of the Krylov schemes under very mild assumptions on the function $\phi$ and a rapid convergence if the function $\phi$ is smooth. 

In terms of the number $m$ of iterations, the three Lanczos methods, and in particular the classical block Lanczos method, turn out to converge considerably faster than the non-adaptive Chebyshev methods. This is indicated in the slightly better error estimates of Section \ref{sec:errorestimates} and Section \ref{sec:errorcbl} and in the numerical experiments provided in Section \ref{sec:experiments}. On the other hand, the three Lanczos methods display a considerably larger memory requirement and a larger computational cost beyond the matrix-vector products. These additional costs are discussed in Section \ref{sec:costs} and have to be taken into account when selecting a particular method for the calculations.     

One final important property of the matrix $\phi(\Ll)$ in the calculation of the kernel predictor is the positive definiteness of the collocation matrices $\Ee_W^* \phi(\Ll) \Ee_W$. This property guarantees the uniqueness of the kernel predictor in the computations. When calculating the approximation of $\phi(\Ll)$, the collocation matrices $\Ee_W^* p_{\phi,m}(\Ll) \Ee_W$ relevant for the kernel machine should therefore possibly inherit this basic property. We will prove in Section \ref{sec:blockKrylov} that this holds generally true for the classical block Lanczos method and can be forced for the Chebyshev method with a workaround. A numerical experiment in Section \ref{sec:experiments} shows that it does not hold true for the other three Krylov methods. We will further see in Section \ref{sec:calculationcbl} that using the classical block Lanczos method the kernel predictor can be calculated without the explicit knowledge of the block $p_{\phi,m}(\Ll) \Ee_W$. 

In Section \ref{sec:kernelmachinesgraphs}, we will now start this work with a brief introduction to kernel methods on graphs. We will shortly summarize some concepts introduced in \cite{erb2022,erb2020} in which the kernel columns $\phi(\Ll) \Ee_W$ of a positive definite kernel $\phi(\Ll) $ have been interpreted as generalized translates of a positive definite graph basis function (GBF). We will further recapitulate how kernel predictors are calculated in a supervised setting with given labels using a regularized least-squares (RLS) approach.

\section{Kernel Machines on Graphs} \label{sec:kernelmachinesgraphs}

We give a short synthesis on discrete kernels and how they are applied in kernel-based learning algorithms for the interpolation, regression and classification on graphs.

\subsection{Graphs and the Graph Laplacian}

Using a simplified setting, we consider simple and undirected graphs $G$ as underlying domains. All required components of a graph $G$ will be encoded in a triplet $G=(V,E,\mathbf{L})$, consisting of $n$ graph vertices $V=\{\node{v}_1, \ldots, \node{v}_{n}\}$, a set $E \subseteq V \times V$ of undirected edges, and a graph Laplacian $\Ll \in \Rr^{n \times n}$.
For kernel-based machine learning on graphs, the graph Laplacian is usually the key ingredient for the construction of the kernels. We suppose that $\Ll$ is a symmetric matrix that encodes the connection weights of the undirected edges of $G$. The entries of $\Ll$ satisfy the following general properties (see \cite[Section 13.9]{GodsilRoyle2001}:
\begin{equation} \label{eq:generalizedLaplacian}
\ds {\begin{array}{ll}\; \Ll_{i,j}<0& \text{if $i \neq j$ and the nodes $\node{v}_{i}, \node{v}_{j}$ are connected}, \\ \; \Ll_{i,j}=0 & \text{if $i\neq j $ and $\node{v}_{i}, \node{v}_{j} $ are not connected}, \\ \; \Ll_{i,i} \in \Rr & \text{for $i \in \{1, \ldots, n\}$}.\end{array}}\end{equation}  

Herein, the negative non-diagonal entries $\Ll_{i,j}$ encode the connection weights between the nodes $\node{v}_{i}$ and $\node{v}_{j}$, while the diagonal elements $\Ll_{i,i}$ provide information about the importance of a single vertex $\node{v}_{i}$. In this work we will, without loss of generality, assume that the Laplacian $\Ll$ is positive semi-definite with all its eigenvalues contained in the interval $[0,\Lambda]$, $\Lambda>0$. An important example is the standard Laplacian $\Ll_S = \mathbf{D} - \mathbf{A}$,  defined in terms of the adjacency matrix $\Aa \in \Rr^{n \times n}$ given by
\begin{equation*}
    \mathbf{A}_{i,j} := 
  \begin{cases}
    1, & \text{if $i \neq j$ and $\node{v}_i, \node{v}_j$ are connected}, \\
    0, & \text{otherwise},
  \end{cases}.
\end{equation*}
and the degree matrix $\mathbf{D}$ with the entries
\begin{equation*}
    \mathbf{D}_{i,j} := 
  \begin{cases}
    \sum_{k=0}^n \mathbf{A}_{i,k}, & \text{if } i=j, \\
    0, & \text{otherwise}.
  \end{cases}
\end{equation*}

\subsection{Positive Definite Kernels on Graphs} \label{sec:pdkernels}

A kernel on the graph $G$ is a function $K : V \times V \to \Rr$ on the Cartesian product $V \times V$ of the vertex set $V$. Linked to the kernel function $K$ is a linear operator $\mathbf{K}: \mathcal{L}(V) \to \mathcal{L}(V)$ acting on the signal space $\mathcal{L}(V) = \{x : V \to \Rr\}$ as 
\[\mathbf{K} x(\node{v}_i) = \sum_{j=1}^n K(\node{v}_i,\node{v}_j) x(\node{v}_j).\] 
By identifying a signal $x \in \mathcal{L}(G)$ with a vector $x = [x(\node{v}_1), \ldots, x(\node{v}_n)]^* \in \Rr^n$, we can represent $\mathbf{K}$ also as the $n \times n$-matrix
\[ \mathbf{K} = \begin{bmatrix} K(\node{v}_1,\node{v}_1) & K(\node{v}_1,\node{v}_2) & \ldots & K(\node{v}_1,\node{v}_n) \\
K(\node{v}_2,\node{v}_1) & K(\node{v}_2,\node{v}_2) & \ldots & K(\node{v}_2,\node{v}_n) \\
\vdots & \vdots & \ddots & \vdots \\
K(\node{v}_n,\node{v}_1) & K(\node{v}_n,\node{v}_2) & \ldots & K(\node{v}_n,\node{v}_n)
\end{bmatrix}.\]
We call the kernel $K$ symmetric positive definite (p.d.) if the corresponding matrix $\mathbf{K}$ is symmetric and positive definite, i.e., if $\mathbf{K}^* = \mathbf{K}$ and 
$$x^{*} \mathbf{K} x  = \sum_{i,j = 1}^n x(\node{v}_i) K(\node{v}_i,\node{v}_j) x(\node{v}_j)  > 0 \quad \text{for all} \ x \in \mathcal{L}(V) \setminus \{0\}.$$ 
The positive definiteness of a kernel guarantees in general the operability of the kernel machines introduced in the next section. 

\subsection{Learning with Kernels on Graphs} \label{sec:learningwithkernels}

Given a training set of $N$ nodes $W = \{\node{w}_{1}, \ldots, \node{w}_{N}\} \subset V$ and respective sampling values $\{y_1, \ldots, y_N\}$ (or labels) a kernel machine aims at finding a regression (or classification) predictor $y$ on the entire node set $V$ in terms of a linear combination of kernel functions 
\begin{equation} \label{eq:representertheorem}
y(\node{v}) = \sum_{i = 1}^N c_i K(\node{v},\node{w}_{i}).
\end{equation}
One important class of kernel machines uses the solution of the linear system
\begin{equation} \label{eq:computationcoefficients} 
 \left( \underbrace{ \begin{bmatrix} K(\node{w}_{1},\node{w}_{1}) & K(\node{w}_{1},\node{w}_{2}) & \ldots & K(\node{w}_{1},\node{w}_{N}) \\
K(\node{w}_{2},\node{w}_{1}) & K(\node{w}_{2},\node{w}_{2}) & \ldots & K(\node{w}_{2},\node{w}_{N}) \\
\vdots & \vdots & \ddots & \vdots \\
K(\node{w}_{N},\node{w}_{1}) & K(\node{w}_{N},\node{w}_{2}) & \ldots & K(\node{w}_{N},\node{w}_{N})
\end{bmatrix}}_{\mathbf{K}_W} + \gamma N \mathbf{I}_N \right)  \begin{bmatrix} c_1 \\ c_2 \\ \vdots \\ c_N \end{bmatrix}
= \begin{bmatrix} y_1 \\ y_2 \\ \vdots \\ y_N \end{bmatrix}
\end{equation} 
to obtain the coefficients $\mathrm{c} = [c_1, \ldots, c_N]^*$ of the predictor $y$. This corresponds to the calculation of a regularized least squares (RLS) solution using a quadratic error term to fit the training data and a regularization term given in terms of the kernel $K$ and a parameter $\gamma > 0$ (see \cite{BelkinMatveevaNiyogi2004,Rifkin2003,Romero2017}). The positive definiteness of the kernel $K$ guarantees that the solution to \eqref{eq:computationcoefficients} is unique. In the limit case $\gamma = 0$, the solution of \eqref{eq:computationcoefficients} yields an interpolating function $y(\node{v})$ in \eqref{eq:representertheorem} that interpolates the data $y(\node{w}_i) = y_i$ at the nodes $\node{w}_i \in W$. For classification purposes, the kernel predictor $y$ can be further processed and, for instance, the values $\mathrm{sign}(y(\node{v}))$ be calculated to classify a node $\node{v}$ in one of two classes $\pm 1$ (assuming that also the labels satisfy $y_i \in \{\pm 1\}$). Another important class of kernel machines is given by the so called support vector machines. In this case, instead of the quadratic error in the RLS solution a hinge loss functional is minimized. For a more profound introduction to support vector machines and kernel-based methods for machine learning we refer to  \cite{Schoelkopf2002,Vapnik1998,Zhu05}.

\subsection{Matrix Functions of the Laplacian and Graph Basis Functions} \label{sec:matrixfunctions}

Most relevant kernels $K$ on a graph $G$ are given as matrix functions $\phi$ of the graph Laplacian $\Ll$. If we assume that $\Ll$ is positive semi-definite with spectrum in the interval $[0,\Lambda]$ and $\phi$ is a positive function on $[0,\Lambda]$, then the kernel matrix $\Kk = \phi(\Ll)$ defined in terms of the spectral decomposition of $\Ll$ is symmetric and positive definite. 

The relevant part of the kernel $K$ for the calculation of the predictor $y$ in \eqref{eq:representertheorem} consists in the $N$ columns $K(\cdot, \node{w}_i)$, $i \in \{1, \ldots, N\}$ of the kernel matrix. We can express these elements in an alternative way by applying, for a node $\node{w} \in W$, the unit signal $e_{\node{w}}$, where 
\begin{equation*}
    e_{\node{w}}(\node{v}) := 
  \begin{cases}
    1, & \text{if } \node{v} = \node{w}, \\
    0, & \text{if } \node{v} \neq \node{w}.
  \end{cases}
\end{equation*}  
As we consider finite ordered node sets $V = \{\node{v}_1, \ldots, \node{v}_n\}$, we can naturally interpret $e_{\node{w}}$ as a canonical unit vector in $\Rr^n$ with value $1$ at the entry $j$ corresponding to the node $\node{w} = \node{v}_j$ and $0$ at all other entries $i \neq j$. Then, the elements $K(\node{v},\node{w}_i)$ can be calculated as
\begin{equation} \label{eq:kernelmatrixfunction} K(\node{v},\node{w}_i) = (\phi(\Ll) e_{\node{w}_i})(\node{v}), \quad i \in \{1, \ldots, N\}.
\end{equation}
The columns $\phi(\Ll) e_{\node{w}_i}$ of the kernel matrix can be interpreted as generalized translates of a positive definite function $f$ on the graph $G$. These generalized translates form a basis for the calculation of the predictor $y$ and the generating function $f$ has been referred to as graph basis function (GBF) in \cite{erb2022,erb2020}. The positive definite GBF $f$ and the function $\phi$ are linked by the graph Fourier transform. Namely, the Fourier coefficients of $f$ correspond to the function $\phi$ evaluated at the eigenvalues of the graph Laplacian $\Ll$. For the details of this relation, we refer to the article \cite{erb2022}. 

One of the most important examples is the exponential function $\phi(\lambda) = e^{- t \lambda}$ with a parameter $t \in \Rr$. The respective kernel on the graph is the well-known diffusion kernel $e^{ -t \mathbf{L}}$ \cite{KondorLafferty2002}. A second prominent example is the variational spline kernel defined as $(\epsilon \mathbf{I}_n + \mathbf{L})^{-s}$ \cite{Pesenson2009,Ward2018interpolating}. Choosing the parameters $\epsilon > 0$ and $s > 0$, this kernel is positive definite. 
  
For general graphs, the matrices $\phi(\Ll)$ are not a priori accessible and the calculation of the spectral decomposition of $\Ll$ is not feasible if the number $n$ of vertices is too large. In order to apply kernel machines using a kernel matrix of the form $\phi(\Ll)$ an efficient calculation of the columns $\phi(\Ll) e_{\node{w}_i}$, $i \in \{1, \ldots, N\}$ is therefore essential. 

\begin{remark}
(i)
An alternative approach to reduce the computational costs for the calculation of a matrix function $\phi(\Ll)$ is to split the graph in smaller subgraphs, using for instance metric clustering techniques as $J$-center clustering \cite{Cavoretto2021,Cavoretto2022} or hierarchical partitioning trees \cite{erb2023}. The single domains of a partitioning are then enlarged to create an overlapping cover of the graph. The main idea of this approach is to calculate the elements of $\phi(\Ll)$ locally on the single subdomains, and then to use a partition of unity to glue the components together. In \cite{Cavoretto2021}, this approach was investigated and particularly for the variational spline it turned out that with increasing overlapping of the domains the merged local kernels converged rapidly towards the global one. Nevertheless, the block Krylov methods studied in this article can also be used as subroutines in \cite{Cavoretto2021,Cavoretto2022} to speed up the local GBF calculations. \\
(ii) Kernels on graphs are not only relevant for regression or classification purposes in kernel machines.
They can also be used to reconstruct, filter and smooth graph signals \cite{Romero2017}, define diffusion wavelets \cite{Coifman2006}, and describe vertex-frequency filters \cite{shuman2016,shuman2020}. They can also be used as tools for the identification of the most influential nodes of a graph, for instance in a social network \cite{Cuomo2023}. Also in these cases, the Krylov algorithms investigated in this work can be applied to accelerate the calculations. 
\end{remark}

\section{Block Krylov subspace methods for fast generation of GBFs} \label{sec:blockKrylov}

In this section, we introduce and review five iterative schemes for the efficient calculation of a block of $N$ matrix-vector products of the form
\begin{equation} \phi(\Ll) \Ee_W, \label{eq:GBFcompute} \end{equation}
where $\Ll \in \Rr^{n \times n}$ is a symmetric positive semi-definite matrix with spectrum in $[0,\Lambda]$, $\phi$ is a positive function on $[0,\Lambda]$, and $\Ee_W \in \Rr^{n \times N}$ is the block vector 
$$\Ee_W = [e_{\node{w}_1}, \ldots, e_{\node{w}_N}] \in \Rr^{n \times N}$$ consisting of $N$ canonical basis vectors $e_{\node{w}_i}$ related to the sampling nodes $\node{w}_i \in W$. Naively, the block vector \eqref{eq:GBFcompute} can be calculated in two steps using the spectral decomposition of the graph Laplacian $\Ll$. With the spectral decomposition of $\Ll$ at hand, the matrix function $\phi(\Ll)$ can be first computed using the functional calculus by evaluating the positive function $\phi$ on the spectrum of $\Ll$. In a second step, the matrix function $\phi(\Ll)$ can then be evaluated on the block $\Ee_W$. Such a procedure is however prohibitive in terms of computational complexity and of memory requirements if the size $n$ of the graph is too large. 

\vspace{1mm}

For an efficient calculation of the basis functions \eqref{eq:GBFcompute} on large graphs $G$, we will therefore make use of the following two principles:
\begin{enumerate}
 \item[(a)] The positive function $\phi$ is approximated with a polynomial $p_{\phi,m}$ of degree $m$.
 \item[(b)] The matrix $\phi(\Ll)$ is never calculated as a whole, only the actions of $\phi(\Ll)$ to the initial block vector $\Ee_W$ are approximated. 
\end{enumerate}
The two construction principles (a) and (b) are profoundly linked with Krylov subspace methods. These principles guarantee that for the calculation of $\phi(\Ll) \Ee_W$ at most $m N$ matrix-vector products are necessary. Furthermore, if the Laplacian $\Ll$ is sparse with at most $r$ nonzero entries in a single row (or column), the complexity of the calculation of $p_m(\Ll) \Ee_W$ is at most of order $\mathcal{O}(r m N n)$.  

\subsection{Classical block Lanczos method for matrix functions} 

Krylov subspace methods use projections into the Krylov spaces
$$ \mathcal{K}_{m}(\Ll,x) = \left\{ \sum_{k=0}^{m-1} c_k \Ll^k x \ : \ c_k \in \Rr \right\}, \quad m \in \Nn, $$
to obtain approximate solutions of eigenvalue problems or linear systems of equations in cases where a direct solution gets too cost-intensive. To approximate a matrix function multiplied with a vector $x \in \Rr^n$, Krylov methods generate a polynomial function $p_{m-1}(\Ll) x$ of order $m-1$ in the Krylov space $\mathcal{K}_{m}(\Ll,x)$ that resembles the matrix-vector product $\phi(\Ll) x$.  

As the matrix function $\phi(\Ll)$ has to be evaluated not only for a single vector $x$ but for an entire block $\Ee_W$ of $N$ unit vectors, we will use block Krylov methods instead. As a first prominent example, we will consider the classical block Lanczos method (see \cite{Frommer2017,Gutknecht2006,Lund2018,Schmelzer2004} for a general introduction to block Krylov methods) with the respective Krylov space given by 
$$\krycl_m(\Ll,\Ee_W) = \left\{ \sum_{k=0}^{m-1} \Ll^k \Ee_W \mathrm{C}_k \ : \ \mathrm{C}_k \in \Rr^{N \times N} \right\}.$$ 
An orthonormal system $\{q_1, \ldots, q_{m N}\} \subset \Rr^n$ of vectors related to the classical Krylov space 
$\krycl_m(\Ll,\Ee_W)$ can be obtained by applying $m - 1 $ steps of a block Lanczos algorithm to the initial block $\mathrm{Q}_1 = [q_1, \ldots, q_N] = \Ee_W$. We store also the remaining basis elements in $n \times N$-blocks $\mathrm{Q}_k$ by setting
\[ \mathrm{Q}_k = [q_{(k-1)N + 1}, \ldots, q_{k N}], \quad k \in \{1, \ldots, m\}. \]
The blocks $\mathrm{Q}_1, \ldots, \mathrm{Q}_m$ in $\krycl_m(\Ll,\Ee_W)$ are determined in such a way that after $m-1$ steps the block Lanczos relation
\begin{equation} \label{eq:blockLanczosrelation} \Ll [\mathrm{Q}_1, \ldots \mathrm{Q}_m] = [\mathrm{Q}_1, \ldots, \mathrm{Q}_{m+1} ] \tilde{\mathbf{H}}_m,
\end{equation}
is satisfied with a block tridiagonal matrix $\tilde{\mathbf{H}}_m \in \Rr^{(m+1) N \times mN}$ of the form

\[ \tilde{\mathbf{H}}_m =
\begin{tikzpicture}[baseline={([yshift=-.5ex]current bounding box.center)}]
 \matrix (vec) [matrix of math nodes, left delimiter = {[}, right delimiter = {]}] {
 \mathrm{H}_{1,1} & \mathrm{H}_{1,2} &    {}    &  {} \\
\mathrm{H}_{2,1} & \mathrm{H}_{2,2} & \ddots & {} \\ 
       {} & \ddots  & \ddots &  \mathrm{H}_{m-1,m} \\ 
       {} &    {}    & \mathrm{H}_{m,m-1} & \mathrm{H}_{m,m}  \\ 
       {} &    {}  &     {}      & \mathrm{H}_{m+1,m}  \\
};
\node (a) at (vec-1-4.north) [above= 10pt, right=30pt]{};
\node (b) at (vec-4-4.south) [right=30pt]{};

\draw [decorate, decoration={brace, amplitude=5pt}] (a) -- (b) node[midway, right=5pt] { $\mathbf{H}_m$.};
\end{tikzpicture}
\]

Here, the blocks $\mathrm{H}_{k+1,k} \in \Rr^{N \times N}$ are upper triangular and invertible and satisfy $\mathrm{H}_{k+1,k} = \mathrm{H}_{k,k+1}^*$, while the blocks $\mathrm{H}_{k,k}$ are symmetric, i.e., $\mathrm{H}_{k,k} = \mathrm{H}_{k,k}^*$. When deleting the last $N$ rows of $\tilde{\mathbf{H}}_m$ we obtain the symmetric block Lanczos matrix $\mathbf{H}_m \in \Rr^{mN \times m N}$. In addition, the classical block Lanczos method enforces the system $\{q_1, \ldots, q_{m N}\}$ to be orthonormal. The generation of the blocks $\mathrm{H}_{k,k}$, $\mathrm{H}_{k+1,k}$ and $\mathrm{Q}_k$ via the block Lanczos iteration is described in Algorithm \ref{algorithm-block-Lanczos}.  

\setlength{\algomargin}{-0.4em}
\begin{algorithm} 
\caption{Classical Block Lanczos algorithm to approximate $\phi(\Ll) \Ee_W$}
\label{algorithm-block-Lanczos}

\begin{multicols}{2}

\begin{algorithmic}[1]
\STATE Set $\mathrm{Q}_1 = \Ee_W$ and $\mathrm{Q}_0 = \mathrm{0}$, $\mathrm{H}_{0,1} = 0$;\\[1mm]
\FOR {$k = 1$ to $m$}
\vspace{1mm}
\STATE $\mathrm{X} = \mathbf{L} \mathrm{Q}_k - \mathrm{Q}_{k-1} \mathrm{H}_{k-1,k}$; 
\STATE $\mathrm{H}_{k,k} = \mathrm{Q}_k^* \, \mathrm{X}$; 
\STATE $\mathrm{X} = \mathrm{X} - \mathrm{Q}_{k} \mathrm{H}_{k,k}$; 
\STATE Compute reduced QR decomposition of $\mathrm{X}$ such that
\[\mathrm{Q}_{k+1} \mathrm{H}_{k+1,k} = \mathrm{X},\]
with $\mathrm{Q}_{k+1} \in \Rr^{n \times N}$ containing $N$ orthonormal columns $q_{k N + 1}$, $\ldots$, $q_{(k+1) N}$ and $\mathrm{H}_{k+1,k} \in \Rr^{N \times N}$ is upper triangular; 
\STATE Set $\mathrm{H}_{k,k+1} = \mathrm{H}_{k+1,k}$;\\[1mm]
\ENDFOR 

\columnbreak

\STATE Set up $\mathbf{H}_m$ from the blocks $\mathrm{H}_{k+1,k}$, $\mathrm{H}_{k,k}$, $k \in \{1,\ldots,m\}$, and calculate
\[ \mathrm{U} = \phi(\mathbf{H}_m) \mathrm{F}_1,\]
where $\mathrm{F}_1 \in \Rr^{m N \times N}$ contains the identity matrix as first $N \times N$-block and all the remaining blocks of $\mathrm{F}_1$ are zero. \\[1mm] 

\STATE {\bfseries Return} 
$p_{\phi,m-1}^{(\mathrm{cbl})}(\Ll) \Ee_W  := [\mathrm{Q}_1, \ldots, \mathrm{Q}_m]\mathrm{U}$ as an approximation to $\phi(\Ll) \Ee_W$.
\end{algorithmic}
\end{multicols}

\end{algorithm}

For the calculation of an approximate kernel predictor, an iterative approach based on the classical block Lanczos method has some advantages. Most importantly, we will show that the linear system \eqref{eq:computationcoefficients} for the coefficients of the kernel predictor with the matrix polynomial $p_{\phi,m-1}^{(\mathrm{cbl})}(\Ll)$ as a kernel has a unique solution. This follows principally from the next theorem. 

\begin{theorem} \label{thm:classicalpositivedefinite}
Assume that the spectrum of $\Ll$ is contained in $[0,\Lambda]$ and that the function $\phi$ is positive on $[0,\Lambda]$. Then the matrix $\Ee_W^* p_{\phi,m-1}^{(\mathrm{cbl})}(\Ll) \Ee_W$ is symmetric and positive definite. 
\end{theorem}

\begin{proof}
Using the same notation as in the description of Algorithm \ref{algorithm-block-Lanczos}, we can rewrite the block vector $p_{\phi,m-1}^{(\mathrm{cbl})}(\Ll) \Ee_W$ as
\[\Ee_W^* p_{\phi,m-1}^{(\mathrm{cbl})}(\Ll) \Ee_W = \Ee_W^*[\mathrm{Q}_1, \ldots, \mathrm{Q}_m] \phi(\mathbf{H}_m) \mathrm{F}_1 = \mathrm{F}_1^*\phi(\mathbf{H}_m) \mathrm{F}_1.\]
In particular, $\Ee_W^* p_{\phi,m-1}^{(\mathrm{cbl})}(\Ll) \Ee_W$ corresponds to the first $N \times N$ principal submatrix of the matrix $\phi(\mathbf{H}_m)$. As $\Ll$ is symmetric, also the block Lanczos matrix $\mathbf{H}_m$ is symmetric. Further, the block Lanczos relation \eqref{eq:blockLanczosrelation} implies the identity 
\[ [q_1, \ldots, q_{mN}]^* \Ll [q_1, \ldots, q_{mN}] = \mathbf{H}_m. \] 
These two properties in combination with the Cauchy interlacing theorem \cite[Section 10.1]{Parlett1987} guarantee that the spectrum of $\mathbf{H}_m$ is contained in the same interval $[0,\Lambda]$ as the spectrum of $\Ll$. Thus, if $\phi$ is positive on $[0,\Lambda]$, the matrix $\phi(\mathbf{H}_m)$ is symmetric and positive definite. Therefore, also the principal submatrix $\Ee_W^* p_{\phi,m-1}^{(\mathrm{cbl})}(\Ll) \Ee_W = \mathrm{F}_1^*\phi(\mathbf{H}_m) \mathrm{F}_1$ is symmetric and positive definite.  
\end{proof}

\subsection{Global block Lanczos method for matrix functions}

The classical block Lanczos method as described in the last section has numerous theoretical advantages and generally requires only a few iterations $m$ to converge. If the block size $N$ is very large, the classical block Lanczos methods displays however some serious drawbacks in terms of memory requirements, the dimensionality of the Krylov space, and the size of the Lanczos matrix $\mathbf{H}_m$. To reduce the dimension of the latter two, an alternative to the classical block method is the usage of the so-called global block Lanczos method first introduced in \cite{Jbilou1999}. 
For the global block Lanczos method, we require the Frobenius inner product and the Frobenius norm of matrix blocks $\mathrm{X}$ and $\mathrm{Y}$ given by
\[ \langle \mathrm{X}, \mathrm{Y} \rangle_F = \sum_{j = 1}^N \sum_{i = 1}^n \mathrm{X}_{i,j} \mathrm{Y}_{i,j} = \tr (\mathrm{Y}^* \mathrm{X}), \quad \|\mathrm{X}\|_F = (\langle \mathrm{X}, \mathrm{X} \rangle_F)^{1/2}.\]
The Krylov space for the global block method is then given by 
$$\krygl_m(\Ll,\Ee_W) = \left\{ \sum_{k=0}^{m-1} c_k \Ll^k \Ee_W  \ : \ c_k \in \Rr \right\},$$
i.e., every element of $\krygl_m(\Ll,\Ee_W)$ is determined as a matrix polynomial of degree $m-1$ applied to the block vector $\Ee_W$. In comparison to the classical case, the space $\krygl_m(\Ll,\Ee_W)$ is spanned by the block vectors $\Ee_W, \Ll \Ee_W, \ldots, \Ll^{m-1} \Ee_W$ and a standard Lanczos algorithm calculates an orthonormal basis $\{\mathrm{Q}_1, \ldots, \mathrm{Q}_m\}$ of this global Krylov space with respect to the Frobenius inner product and starting with the initial vector $\Ee_W$. In this way, the respective Lanczos matrix $\mathbf{H}_m$ is just a $m \times m$ tridiagonal matrix with real-valued entries $h_{k,k}, h_{k,k+1}  \in \Rr$ instead of blocks $\mathrm{H}_{k,k}, \mathrm{H}_{k,k+1} \in \Rr^{N \times N}$. The global block Lanczos orthogonalization procedure is summarized in Algorithm \ref{algorithm-global-block-Lanczos}. In comparison to the classical block method, we can not guarantee that the submatrix
$\Ee_W^* p_{\phi,m-1}^{(\mathrm{gbl})}(\Ll) \Ee_W$ is positive semi-definite for the global method. We can however guarantee its symmetry. The following proposition follows simply by the fact that for a symmetric matrix also the matrix polynomial $p_{\phi,m-1}^{(\mathrm{gbl})}(\Ll)$ is symmetric.  

\begin{proposition}
The matrix $\Ee_W^* p_{\phi,m-1}^{(\mathrm{gbl})}(\Ll) \Ee_W \in \Rr^{N \times N}$ is symmetric.   
\end{proposition}

\begin{algorithm} 
\caption{Global Block Lanczos algorithm to approximate $\phi(\Ll) \Ee_W$}
\label{algorithm-global-block-Lanczos}

\begin{multicols}{2}

\begin{algorithmic}[1]
\STATE Set $\mathrm{Q}_1 = \Ee_W/\sqrt{N}$ and $\mathrm{Q}_0 = \mathrm{0}$, $h_{0,1} = 0$;\\[3mm]
\FOR {$k = 1$ to $m$} %\\[2mm]
\vspace{1mm}
\STATE $\mathrm{X} = \mathbf{L} \mathrm{Q}_k - \mathrm{Q}_{k-1} h_{k-1,k}$; \\[1mm]
\STATE $h_{k,k} = \tr (\mathrm{Q}_k^* \, \mathrm{X})$; \\[1mm]
\STATE $\mathrm{X} = \mathrm{X} - \mathrm{Q}_{k} h_{k,k}$; \\[1mm]
\STATE Compute norm $h_{k+1,k} = \| \mathrm{X} \|_F$ and set 
$\mathrm{Q}_{k+1} = \mathrm{X}/h_{k+1,k}$; \\[1mm]
\STATE Set $h_{k,k+1} = h_{k+1,k}$;\\[1mm]
\ENDFOR 

\columnbreak 

\STATE Set up $\mathbf{H}_m$ from the real numbers $h_{k+1,k}$, $h_{k,k}$, $k \in \{1,\ldots,m\}$, and calculate the vector
\[ u = \sqrt{N} \phi(\mathbf{H}_m) f_1,\]
with $f_1 \in \Rr^{m}$ being the first canonical vector in $\Rr^{m}$;\\[1mm] 

\STATE {\bfseries Return} 
\begin{align*}
p_{\phi,m-1}^{(\mathrm{gbl})}(\Ll) \Ee_W  &:= \sum_{k=1}^m u_k \mathrm{Q}_k \\ &= [\mathrm{Q}_1, \ldots, \mathrm{Q}_m] (u \otimes \mathbf{I}_N)
\end{align*} as an approximation to $\phi(\Ll) \Ee_W$.
\end{algorithmic}

\end{multicols}

\end{algorithm}

\subsection{Sequential Lanczos method for matrix functions}

For the classical as well as for the global block Lanczos method the matrix $[\mathrm{Q}_1, \ldots, \mathrm{Q}_m]$ generated during the Lanczos process has to be stored in order to calculate the approximation $\phi(\Ll) \Ee_W$. For large dimensions $n$ and block sizes $N$ this might not be feasible. One simple possibility to avoid the dependency of the storage costs on the block size is to use a sequential Lanczos method in which an ordinary non-block Lanczos method is applied independently to each single column $e_{\node{w}_i}$ of $\Ee_W$. This sequential procedure generates approximations in the Krylov spaces
$$\krysl_m(\Ll,\Ee_W) = \mathcal{K}_{m}(\Ll,e_{\node{w}_1}) \times \cdots \times \mathcal{K}_{m}(\Ll,e_{\node{w}_N}) = \left\{ \sum_{k=0}^{m-1} \Ll^k \Ee_W c_k \ : \ c_k \in \Rr^{N} \right\}.$$
The respective approximant $p_{\phi,m-1}^{(\mathrm{sbl})}(\Ll) \Ee_W$ for $\phi(\Ll) \Ee_W$ is given as 
\[p_{\phi,m-1}^{(\mathrm{sbl})}(\Ll) \Ee_W = [p_{\phi,m-1}(\Ll)e_{\node{w}_1}, \ldots, p_{\phi,m-1}(\Ll)e_{\node{w}_N} ], \]
where 
$$p_{\phi,m-1}(\Ll)e_{\node{w}_i} := [q_1^{(\node{w}_i)}, \ldots, q_m^{(\node{w}_i)}] \mathbf{H}_m^{(\node{w}_i)} f_1$$ denotes the result of an ordinary Lanczos method after $m$ steps to obtain a matrix function applied to the vector $e_{\node{w}_i}$. This non-block Lanczos method corresponds precisely to the output of 
Algorithm \ref{algorithm-block-Lanczos} or Algorithm \ref{algorithm-global-block-Lanczos} with $N=1$ applied to a single column $e_{\node{w}_i}$ of the block $\Ee_W$. For the sequential Lanczos method, we will use the tridiagonal matrix 
\[\mathbf{H}_m = \operatorname{diag}(\mathbf{H}_m^{(\node{w}_1)}, \ldots, \mathbf{H}_m^{(\node{w}_N)}) \in \Rr^{mN \times mN}\]
as a unified Lanczos matrix for the entire block. In this way, we can write the matrix-vector product $p_{\phi,m-1}^{(\mathrm{sbl})}(\Ll) \Ee_W$ alternatively also as
\[p_{\phi,m-1}^{(\mathrm{sbl})}(\Ll) \Ee_W = [q_1^{\node{w}_1}, \ldots, q_m^{\node{w}_1}, \ldots, q_1^{\node{w}_N}, \ldots, q_m^{\node{w}_N} ] \mathbf{H}_m (\mathbf{I}_N \otimes f_1 ). \]

\subsection{Chebyshev polynomial approximation for matrix functions} \label{sec:chebyshev}

A final simple approach to obtain a polynomial approximation of the block $\phi(\Ll) \Ee_W$ is given by the approximation of the function $\phi$ in terms of Chebyshev polynomials. 
As before, we assume that the spectrum of the graph Laplacian $\Ll$ is contained in the interval $[0,\Lambda]$. Then, we can approximate the function $\phi$ by creating a polynomial interpolant of degree $m$ based on function values of $\phi$ on a dilated and shifted Chebyshev-Lobatto grid $\{ \frac{\Lambda}{2} (1 - \cos(\frac{\pi j}{m})) \, : \, j \in \{0, \ldots, m\}\}$. This provides an approximation of the function $\phi$ of the form 
\begin{equation} \label{eq:apprchebyshevseries} p_{\phi,m}^{(\mathrm{cheb})}(\lambda) =  \sum_{k=0}^m c_k(\phi) T_k\left( 1 -  \ts \frac{ 2}{\Lambda} \lambda\right),\end{equation}
where $T_k(\lambda) = \cos(k \arccos(\lambda))$, $k \in \Nn_0$, denote the Chebyshev polynomials of the first kind on the interval $[-1,1]$. The coefficients $c_k(\phi)$ are computed as
\begin{equation} \label{eq:chebcoeff} c_k(\phi) = \frac{\phi(0) + (-1)^k \phi(\Lambda)}{m} + \frac{2}{m}\sum_{j = 0}^m \phi\left( \ts \frac{\Lambda}{2} (1 - \cos(\frac{\pi j}{m}))\right) \cos\left(\frac{\pi j k }{m} \right)  \end{equation}
and require only the knowledge of the function $\phi$ on the altered Chebyshev grid $\{ \frac{\Lambda}{2} (1 - \cos(\frac{\pi j}{m})) \, : \, j \in \{0, \ldots, m\}\} \subset [0,\Lambda]$. The calculation of the sum in \eqref{eq:chebcoeff} can be performed efficiently by a fast Fourier or a fast cosine transform. Once the coefficients $c_k(\phi)$ are calculated, the Chebyshev approximation $p_{\phi,m}^{(\mathrm{cheb})}(\Ll) \Ee_W$ of the matrix function $\phi(\Ll) \Ee_W$ can be evaluated in terms of matrix-vector products and simple vector space operations by using the three-term recurrence relation
\[T_{k+1}(\lambda) = 2 \lambda T_k(\lambda) - T_{k-1}(\lambda), \quad T_1(\lambda) = 1, \; T_0(\lambda) = 1,\]
of the Chebyshev polynomials $T_k$ to generate the blocks $T_k(\mathbf{I}_n - \ts \frac{2}{\Lambda}\Ll) \Ee_W$. Compared to the block Lanczos methods discussed in the previous sections the Chebyshev method requires no memory to store an orthogonal basis of the Krylov space. On the other hand, the convergence of $p_{\phi,m}^{(\mathrm{cheb})}(\Ll) \Ee_W$ towards $\phi(\Ll) \Ee_W$ is in general slower compared to the discussed block Lanczos methods. This fact is well-known for classical Lanczos methods applied to a single vector (see for instance \cite{Bergamaschi2000}), but will also get theoretically and practically apparent for blocks in the upcoming sections.  

\subsubsection{Enforcing positive semi-definiteness of the approximation} The definition in \eqref{eq:apprchebyshevseries} guarantees that the matrix $p_{\phi,m}^{(\mathrm{cheb})}(\Ll)$ is symmetric . It does however not ensure that the Chebyshev interpolation polynomial $p_{\phi,m}^{(\mathrm{cheb})}(\lambda)$ is non-negative on $[0,\Lambda]$, and, thus, that $p_{\phi,m}^{(\mathrm{cheb})}(\Ll)$ is a positive semi-definite matrix. Our numerical experiments in Section \ref{sec:experiments} will in fact show that the matrices $p_{\phi,m}^{(\mathrm{cheb})}(\Ll)$ can have negative eigenvalues. In order to enforce positive semi-definiteness, we can however make use of the following simple workaround: we approximate the square root $\sqrt{\phi}$ by the polynomial $p_{\phi^{1/2},\lfloor m/2 \rfloor}^{(\mathrm{cheb})}$ and get then the non-negative polynomial 
$$p_{\phi,m}^{(\mathrm{cheb}^2)}(\lambda) = \left(p_{\phi^{1/2},\lfloor m/2 \rfloor}^{(\mathrm{cheb})}(\lambda)\right)^2$$ 
as an approximation of the function $\phi$. While this ensures that the matrix $p_{\phi,m}^{(\mathrm{cheb}^2)}(\Ll)$ is positive semi-definite, the convergence of $p_{\phi,m}^{(\mathrm{cheb}^2)}$ towards $\phi$ is in general slower than the convergence of the interpolation polynomial $p_{\phi,m}^{(\mathrm{cheb})}$ towards $\phi$. This will be visible in the numerical experiments at the end of this manuscript.

\section{Error estimates} \label{sec:errorestimates}

In the following, we provide common error estimates for all considered block Krylov methods. Central for the three block Lanczos methods is the following auxiliary result.

\begin{lemma} \label{lem:errorkrylov}
Let $\Ll \in \Rr^{n \times n}$ be symmetric, and $p_{m-1}$ be any polynomial of degree $m-1$, approximating $\phi(\lambda)$ on the interval $[0,\Lambda]$ with the residual $r_{m}(\lambda)=\phi(\lambda) - p_{m-1}(\lambda)$. Then, we get for all three block Lanczos methods $\mathrm{kr}\in \{\mathrm{cbl}, \mathrm{gbl}, \mathrm{sbl}\}$
$$
\left\| \phi(\Ll) \Ee_W - p_{\phi,m-1}^{(\mathrm{kr})}(\Ll) \Ee_W \right\|_F \leq \sqrt{N} \left( \|r_{m}(\Ll)\|_2+ \|r_{m}(\mathbf{H}_m)\|_2\right).
$$
\end{lemma}

\begin{proof}
We prove this result for $\mathrm{kr} = \mathrm{cbl}$, i.e., for the classical block Lanczos method. We thereby follow an argumentation line developed in \cite[Lemma A.1]{Gallopoulos1992} for the calculation of the matrix exponential $e^{-\Ll} x$ using an ordinary Lanczos method. For the global and the sequential block Lanczos schemes, the proof is, up to some minor modifications regarding technical particularities, the same. 

\vspace{1mm}

By definition of $p_{m-1}$ and $r_m$, we have $\phi(\lambda) = p_{m-1}(\lambda)+r_m(\lambda)$, and thus 
\begin{equation} \label{eq:proof1}
\phi(\Ll) \Ee_W = p_{m-1}(\Ll) \Ee_W + r_{m}(\Ll) \Ee_W.
\end{equation}
With an induction argument and the block Lanczos relation \eqref{eq:blockLanczosrelation} for the classical regime, we further have the identity 
$$\Ll^j \Ee_W = [\mathrm{Q}_1, \ldots, \mathrm{Q}_m] \mathbf{H}_m^j \mathrm{F}_1, \quad j \in \{0, \ldots, m-1\}.$$ 
The same identity holds then true for all polynomials $p_{m-1}$ of degree $m-1$, i.e.,
$$
p_{m-1}(\Ll) \Ee_W = [\mathrm{Q}_1, \ldots, \mathrm{Q}_m] p_{m-1} \left(\mathbf{H}_m\right) \mathrm{F}_1.
$$
By the relation of $p_{m-1}$ with the residual $r_m$, we can write
$$
p_{m-1}\left(\mathbf{H}_m\right) \mathrm{F}_1 = \phi(\mathbf{H}_m) \mathrm{F}_1 - r_m \left(\mathbf{H}_m\right) \mathrm{F}_1.
$$
Substituting the latter two identities in equation \eqref{eq:proof1}, we get
\begin{equation*}
\phi(\Ll) \Ee_W =  [\mathrm{Q}_1, \ldots, \mathrm{Q}_m] \phi(\mathbf{H}_m) \mathrm{F}_1 +
r_m(\Ll) \Ee_W - [\mathrm{Q}_1, \ldots, \mathrm{Q}_m] r_m\left(\mathbf{H}_m\right) \mathrm{F}_1. 
\end{equation*}
Finally, applying the Frobenius norm, we get the estimate
\begin{align*}
\| \phi(\Ll) \Ee_W &-  [\mathrm{Q}_1, \ldots, \mathrm{Q}_m] \phi(\mathbf{H}_m) \mathrm{F}_1 \|_F =
\| r_m(\Ll) \Ee_W - [\mathrm{Q}_1, \ldots, \mathrm{Q}_m] r_m\left(\mathbf{H}_m\right) \mathrm{F}_1 \|_F \\ & \leq \| r_m(\Ll) \Ee_W \|_F + \| [\mathrm{Q}_1, \ldots, \mathrm{Q}_m] r_m\left(\mathbf{H}_m\right)  \mathrm{F}_1 \|_F \\
&\leq \| r_m(\Ll)\|_2 \| \Ee_W \|_F + \| [\mathrm{Q}_1, \ldots, \mathrm{Q}_m] r_m\left(\mathbf{H}_m\right)\|_2  \| \mathrm{F}_1 \|_F \\
&\leq \sqrt{N} \| r_m(\Ll)\|_2 + \sqrt{N} \| r_m\left(\mathbf{H}_m\right)\|_2.
\end{align*}
In these last steps, we just used elementary properties of the Frobenius norm as the triangle inequality, the compatibility property $\|\mathbf{A} \mathbf{B} \|_F \leq \|\mathbf{A}\|_2 \|\mathbf{B}\|_F$ and $\|\Ee_W\|_F = \|\mathrm{F}_1\|_F = \sqrt{N}$. 
\end{proof}

\begin{theorem} \label{thm:errorestimates}
Let $\Ll \in \Rr^{n \times n}$ be symmetric with spectrum in $[0,\Lambda]$. Then, we get for all three block Lanczos methods $\mathrm{kr}\in \{\mathrm{cbl}, \mathrm{gbl}, \mathrm{sbl}\}$ the estimate
$$
\left\| \phi(\Ll) \Ee_W - p_{\phi,m}^{(\mathrm{kr})}(\Ll) \Ee_W \right\|_F \leq 2 \sqrt{N} E_{m}(\phi),
$$
where
\[E_{m}(\phi) = \min_{p \in \Pi_{m}} \max_{\lambda \in [0,\Lambda]}|\phi(\lambda) - p(\lambda)|\]
denotes the best approximation error for the function $\phi$ in the space of polynomials $\Pi_{m}$ of degree less or equal to $m$ on the interval $[0,\Lambda]$. On the other hand, for the Chebyshev method, we have the bound  
$$
\left\| \phi(\Ll) \Ee_W - p_{\phi,m}^{(\mathrm{cheb})}(\Ll) \Ee_W \right\|_F \leq \sqrt{N} \left(2 + \textstyle \frac{2}{\pi} \log (m+1) \right) E_{m}(\phi).
$$
\end{theorem}

\begin{proof}
For the three block Lanczos methods $\mathrm{kr}\in \{\mathrm{cbl}, \mathrm{gbl}, \mathrm{sbl}\}$, we can use the result of Lemma \ref{lem:errorkrylov} together with the fact that $\Ll$ is symmetric. As the matrix $\mathbf{H}_{m+1}$ is a representation of the projection of $\Ll$ into the Krylov space $\krykr_{m+1}(\Ll,\Ee_W)$, the Cauchy interlacing theorem \cite[Section 10.1]{Parlett1987}) guarantees that the spectrum of $\mathbf{H}_{m+1}$ is also contained in the interval $[0,\Lambda]$. Therefore, Lemma \ref{lem:errorkrylov} and the properties of the spectral matrix norm yield
\begin{align*} \left\| \phi(\Ll) \Ee_W - p_{\phi,m}^{(\mathrm{kr})}(\Ll) \Ee_W \right\|_F 
&\leq \sqrt{N} \left( \|r_{m+1}(\Ll)\|_2+ \|r_{m+1}(\mathbf{H}_{m+1})\|_2\right)  \\
& =  \sqrt{N} \left( \max_{\lambda \; \text{eig of $\Ll$}} |r_{m+1}(\lambda)|+ \max_{\lambda \; \text{eig of $\mathbf{H}_{m+1}$}} |r_{m+1}(\lambda)|\right)\\
&\leq 2 \sqrt{N} \max_{\lambda \in [0,\Lambda]} |\phi(\lambda) - p_m(\lambda)|
\end{align*}
for every polynomial $p_m$ of degree less or equal to $m$. This gives the first statement of the theorem. The second statement about the Chebyshev method follows from well-known error estimates for the polynomial interpolant $p_{\phi,m}^{(\mathrm{cheb})}(\lambda)$ on the Chebyshev-Lobatto grid \cite[Theorem 16.1]{Trefethen2013}:
\begin{align*} \left\| \phi(\Ll) \Ee_W - p_{\phi,m}^{(\mathrm{cheb})}(\Ll) \Ee_W \right\|_F 
&\leq \left\| \phi(\Ll) \Ee_W - p_{\phi,m}^{(\mathrm{cheb})}(\Ll) \right\|_2 \left\| \Ee_W \right\|_F  \\
& = \sqrt{N} \max_{\lambda \; \text{eig of $\Ll$}} |\phi(\lambda) - p_{\phi,m}^{(\mathrm{cheb})}(\lambda)|\\
&\leq \sqrt{N} \max_{\lambda \in [0,\Lambda]} |\phi(\lambda) - p_{\phi,m}^{(\mathrm{cheb})}(\lambda)| \\
&\leq  \sqrt{N} \left(2 + \textstyle \frac{2}{\pi} \log (m+1) \right) E_{m}(\phi).
\end{align*}
\end{proof}

Depending on the particular characteristics of the function $\phi$, there are numerous more or less explicit estimates for the best polynomial approximation $E_m(\phi)$ in the literature. We provide two classical examples from univariate approximation. 

\begin{example}
If the function $\phi$ has $r$ bounded derivatives such that $|\phi^{(r)}(\lambda)| \leq M_{r}$ for all $\lambda \in [0,\Lambda]$, then the best polynomial approximation can be bounded by
\[ E_m(\phi) \leq \frac{C_r \Lambda^r M_r }{m^r},\]
with a constant $C_r$ that depends only on $r$ \cite[Chapter VI, Section 2, Corollary 2]{Natanson1964}.
\end{example} 

\begin{example}
As a second example, we assume that $\phi$ is an analytic function on $[0,\Lambda]$ and analytically continuable to the open Bernstein ellipse $E_{\rho}$ with foci $\{0,\Gamma\}$ and sum of the half axes equal to $\frac{\Lambda}{2} \rho$. Further, we suppose that $\phi$ is bounded on $E_{\rho}$ by $|\phi(\lambda)| \leq M$. Then, the best polynomial approximation can be bounded by \cite[Chapter 7, Section 8]{DeVore1993}.
\[ E_m(\phi) \leq \frac{2 M}{ \rho - 1} \rho^{-m}.\]
With the same assumptions on the function $\phi$, a similar estimate (with an additional factor $2$) is obtainable for the uniform error $\|\phi - p_{\phi,m}^{(\mathrm{cheb})}\|_\infty$ of the Chebyshev polynomial approximation considered in Section \ref{sec:chebyshev}, cf. \cite[Theorem 8.2]{Trefethen2013}.
\end{example} 

\begin{example}
For the exponential function $\phi(\lambda) = e^{-t \lambda}$, $t > 0$, an explicit upper bound for the error $E_m(e^{-t \lambda})$ is given in \cite[Theorem 3]{stewartleyk1996}. This bound reads as
\begin{equation} \label{eq:boundapproximation} E_m(e^{-t \lambda}) \leq \left\{ \begin{array}{ll} 2 e^{- \frac{b (m+1)^2}{t\Lambda}} \big( 1 + \ts \sqrt{\frac{t\Lambda \pi}{4b}}\big)+ 2 \frac{ d^{t\Lambda}}{1-d}& \text{if $m \leq t \Lambda$,} \\ 2 \frac{d^m}{1-d} & \text{if $m > t \Lambda$.}\end{array} \right. \end{equation}
The two constants $b$ and $d$ in this bound are explicitly known: $b = (\sqrt{5}-1)/2 \approx 0.618$ and $d = (\sqrt{5}-2) e^b \approx 0.438$. This explicit bound is also useful as a criterion for the choice of the degree $m$ in calculation of the matrix exponential with Chebyshev polynomials (cf. \cite{Bergamaschi2000}).
\end{example}

\section{Error estimates for RLS kernel predictors} \label{sec:errorcbl}

In the previous section we saw that, under some mild assumptions on the function $\phi$, the block Krylov iterates $p_{\phi,m-1}^{(\mathrm{kr})}(\Ll) \Ee_W$ approximate the kernel component $\phi(\Ll) \Ee_W$ as $m$ gets large. Now, if we consider the single columns of $\phi(\Ll) \Ee_W$ as basis vectors of a kernel machine, we see that also linear combinations of vectors in $\phi(\Ll) \Ee_W$ are approximated by respective linear combinations of elements of $p_{\phi,m-1}^{(\mathrm{kr})}(\Ll) \Ee_W$. In the light of the RLS kernel predictors introduced in Section \ref{sec:learningwithkernels}, we can therefore also consider the approximate RLS predictors 
\[ y^{(\mathrm{kr})} = \sum_{i=1}^N c^{(\mathrm{kr})}_i p_{\phi,m-1}^{(\mathrm{kr})}(\Ll) \Ee_W,\]
based on the coefficients $c^{(\mathrm{kr})}_i$ determined as solutions of the linear system
\begin{equation} \label{eq:computationcoefficientskrylov} 
\Big( \Ee_W^* p_{\phi,m-1}^{(\mathrm{kr})}(\Ll) \Ee_W + \gamma N \mathbf{I}_N\Big) 
\underbrace{\begin{bmatrix} c_1^{(\mathrm{kr})} \\ c_2^{(\mathrm{kr})} \\ \vdots \\ c_N^{(\mathrm{kr})} \end{bmatrix}}_{\mathrm{c}^{(\mathrm{kr})}}
= \underbrace{\begin{bmatrix} y_1 \\ y_2 \\ \vdots \\ y_N \end{bmatrix}}_{\mathrm{y}}.
\end{equation}
For this, we can expect convergence of the Krylov predictor $y^{(\mathrm{kr})}$ towards the original predictor $y$ as $m$ gets large. This is specified in the following theorem. 

\begin{theorem} \label{thm:errorestimates2}
Let $\mathrm{kr} \in \{\mathrm{cbl}, \mathrm{gbl}, \mathrm{sbl}, \mathrm{cheb}, \mathrm{cheb}^2\}$ and $\phi$ be continuous and positive on $[0,\Lambda]$ with 
$$\phi_{\min} = \min_{\lambda \in [0,\Lambda]} |\phi(\lambda)| \quad \text{and} \quad \phi_{\max} = \max_{\lambda \in [0,\Lambda]} |\phi(\lambda)|.$$ We further suppose that
$p_{\phi,m-1}^{(\mathrm{kr})}(\Ll) \Ee_W$ converges to $\phi(\Ll) \Ee_W$ as $m \to \infty$ and that $\gamma \geq 0$.
Then, for $m \to \infty$, we get the asymptotic bound
\[ \|y - y^{(\mathrm{kr})}\|_2 \dot{\leq} 
\frac{\|\mathrm{y}\|_2}{\phi_{\min} + \gamma N} \left(1 + \frac{\phi_{\max}}{\phi_{\min} + \gamma N}\right) \left\| \phi(\Ll) \Ee_W - p_{\phi,m-1}^{(\mathrm{kr})}(\Ll) \Ee_W \right\|_2.  \]
\end{theorem}

\begin{proof}
For an invertible matrix $\Aa \in \Rr^{n \times n}$ we have the upper bound
\begin{equation} \label{eq:1} \| \Aa^{-1} - (\Aa + \Bb)^{-1}\|_2 \leq \|\Bb\|_2 \|\Aa^{-1}\|_2^2 + \mathcal{O}(\|\Bb\|_2^2), 
\end{equation}
for small enough perturbation matrices $\Bb \in \Rr^{n \times n}$. Writing the predictors $y^{(\mathrm{kr})}$ in terms of the solution of the linear system \eqref{eq:computationcoefficientskrylov} and using the triangle inequality, we get
\begin{align*}
\|y& - y^{(\mathrm{kr})}\|_2 = \| \phi(\Ll) \Ee_W (\Ee_W^* \phi(\Ll) \Ee_W + \gamma N \mathbf{I}_N)^{-1} \mathrm{y} \\ & \qquad \qquad \qquad -  p_{\phi,m-1}^{(\mathrm{kr})}(\Ll) \Ee_W (\Ee_W^* p_{\phi,m-1}^{(\mathrm{kr})}(\Ll) \Ee_W + \gamma N \mathbf{I}_N)^{-1} \mathrm{y} \|_2 \\
& \leq \| \phi(\Ll) \Ee_W ((\Ee_W^* \phi(\Ll) \Ee_W + \gamma N \mathbf{I}_N)^{-1}-(\Ee_W^* p_{\phi,m-1}^{(\mathrm{kr})}(\Ll) \Ee_W + \gamma N \mathbf{I}_N)^{-1}) \mathrm{y}\|_2  \\ & \qquad + \|(p_{\phi,m-1}^{(\mathrm{kr})}(\Ll) -  \phi(\Ll)) \Ee_W (\Ee_W^* p_{\phi,m-1}^{(\mathrm{kr})}(\Ll) \Ee_W + \gamma N \mathbf{I}_N)^{-1} \mathrm{y}  \|_2.
\end{align*} 
For the first term of this inequality, the bound \eqref{eq:1} implies asymptotically for large $m$ the upper estimate
\begin{align*}
\| \phi(\Ll)& \Ee_W ((\Ee_W^* \phi(\Ll) \Ee_W + \gamma N \mathbf{I}_N)^{-1}-(\Ee_W^* p_{\phi,m-1}^{(\mathrm{kr})}(\Ll) \Ee_W + \gamma N \mathbf{I}_N)^{-1}) \mathrm{y} \|_2 \\ 
& \dot{\leq}  \| \phi(\Ll) \Ee_W\|_2 \|(\Ee_W^* (\phi(\Ll) - p_{\phi,m-1}^{(\mathrm{kr})}(\Ll)) \Ee_W  \|_2\|(\Ee_W^* \phi(\Ll) \Ee_W + \gamma N \mathbf{I}_N)^{-1}\|_2^2\| \mathrm{y} \|_2 \\
&\dot{\leq} \frac{\phi_{\max} \|\mathrm{y}\|_2}{(\phi_{\min}+\gamma N)^2} \left\| \phi(\Ll) \Ee_W - p_{\phi,m-1}^{(\mathrm{kr})}(\Ll) \Ee_W \right\|_2
\end{align*} 
For the second term, we expand the term $\Ee_W^* p_{\phi,m-1}^{(\mathrm{kr})}(\Ll) \Ee_W$ as $\Ee_W^* p_{\phi,m-1}^{(\mathrm{kr})}(\Ll) \Ee_W = \Ee_W^* (p_{\phi,m-1}^{(\mathrm{kr})}(\Ll) -\phi(\Ll)) \Ee_W  +\Ee_W^*\phi(\Ll) \Ee_W$. Then, we apply the triangle inequality and again the bound given in \eqref{eq:1}. In this way, we get asymptotically for large $m$ the estimate
\begin{align*}
\|(p_{\phi,m-1}^{(\mathrm{kr})}(\Ll)& -  \phi(\Ll)) \Ee_W (\Ee_W^* p_{\phi,m-1}^{(\mathrm{kr})}(\Ll) \Ee_W + \gamma N \mathbf{I}_N)^{-1} \mathrm{y}  \|_2 \\ 
&\dot{\leq} \frac{\|\mathrm{y}\|_2}{(\phi_{\min}+\gamma N)} \left\| \phi(\Ll) \Ee_W - p_{\phi,m-1}^{(\mathrm{kr})}(\Ll) \Ee_W \right\|_2.
\end{align*} 
\end{proof}

\begin{remark}
As $\| \phi(\Ll) \Ee_W - p_{\phi,m-1}^{(\mathrm{kr})}(\Ll) \Ee_W \|_2 \leq \| \phi(\Ll) \Ee_W - p_{\phi,m-1}^{(\mathrm{kr})}(\Ll) \Ee_W \|_F$, we can combine the estimates of Theorem \ref{thm:errorestimates2} and Theorem \ref{thm:errorestimates} to get an estimate of $\|y - y^{(\mathrm{kr})}\|_2$ in terms of the best polynomial error $E_m(\phi)$. In this way, we get for $\mathrm{kr} \in \{\mathrm{cbl}, \mathrm{gbl}, \mathrm{sbl}\}$ the error estimate
\[ \|y - y^{(\mathrm{kr})}\|_2 \dot{\leq} 
\frac{2 \sqrt{N} \|\mathrm{y}\|_2}{\phi_{\min} + \gamma N} \left(1 + \frac{\phi_{\max}}{\phi_{\min} + \gamma N}\right) E_{m-1}(\phi).  \]
\end{remark}

\section{Calculating kernel predictors with the classical block Lanczos method} \label{sec:calculationcbl}

The columns $\{q_1, \ldots, q_{m N}\}$ of the blocks $\mathrm{Q}_1, \ldots, \mathrm{Q}_m$ delivered by  the classical block Lanczos method form an orthonormal system. This is in general not the case for the global and the sequential block Lanczos method. Also in view of the calculation of the kernel predictors the classical block Lanczos method has some theoretic advantages compared to the other two block methods. Most importantly, for the classical block method we can guarantee that the linear system \eqref{eq:computationcoefficientskrylov} has a unique solution, i.e., that the RLS predictor $y^{(\mathrm{cbl})}$ is always uniquely determined. This is an immediate consequence of Theorem  \ref{thm:classicalpositivedefinite}.

\begin{corollary} \label{cor:1}
Let $\phi$ be a positive function on the interval $[0,\Lambda]$ which contains the spectrum of $\Ll$. Then, for every $m \geq 1$ and $\gamma \geq 0$ there exists a unique RLS kernel predictor $y^{(\mathrm{cbl})}$ defined upon \eqref{eq:computationcoefficientskrylov} by using the classical block Lanczos method for the generation of $p_{\phi,m-1}^{(\mathrm{cbl})}(\Ll) \Ee_W$. The predictor $y^{(\mathrm{cbl})}$ can be written as a linear combination of the vectors $\{q_1, \ldots, q_{m N}\}$ as 
\[ y^{(\mathrm{cbl})} = [\mathrm{Q}_1, \ldots, \mathrm{Q}_m] \phi(\mathbf{H}_m) \mathrm{F}_1 \mathrm{c}^{(\mathrm{cbl})}.\]
\end{corollary}

\begin{proof}
By Theorem \ref{thm:classicalpositivedefinite}, we know that the matrix $\Ee_W^* p_{\phi,m-1}^{(\mathrm{cbl})}(\Ll) \Ee_W$ is positive definite for all $m \geq 1$. This implies that the linear system \eqref{eq:computationcoefficientskrylov} has a unique solution for all $m \geq 1$ and $\gamma \geq 0$. Furthermore, the classical block Lanczos iterations as described in Algorithm \ref{algorithm-block-Lanczos} lead to the identity
\[y^{(\mathrm{cbl})} = p_{\phi,m-1}^{(\mathrm{cbl})}(\Ll) \Ee_W \mathrm{c}^{(\mathrm{cbl})} = [\mathrm{Q}_1, \ldots, \mathrm{Q}_m] \phi(\mathbf{H}_m) \mathrm{F}_1 \mathrm{c}^{(\mathrm{cbl})}.\] 
\end{proof}

While proving Theorem  \ref{thm:classicalpositivedefinite}, we have shown that $\Ee_W^* p_{\phi,m-1}^{(\mathrm{cbl})}(\Ll) \Ee_W = \mathrm{F}_1^*\phi(\mathbf{H}_m) \mathrm{F}_1$ holds true. This identity allows to calculate the coefficients $\mathrm{c}^{(\mathrm{cbl})}$ directly by using the matrix $\phi(\mathbf{H}_m)$. In particular, the explicit knowledge of the block $p_{\phi,m-1}^{(\mathrm{cbl})}(\Ll) \Ee_W$ is not required to calculate the predictor $y^{(\mathrm{cbl})}$. We summarize all steps for the calculation of $y^{(\mathrm{cbl})}$ in Algorithm \ref{algorithm-RLS-Krylov}. 

\begin{algorithm} 
\caption{Kernel RLS predictor using classical block Lanczos method}
\label{algorithm-RLS-Krylov}

\begin{algorithmic}[1]
\STATE \textbf{Input:} Labels $y_1, \ldots, y_N$ at the sampling nodes $W = \{\node{w}_{1}, \ldots, \node{w}_{N}\} \subset V$. \\
A graph Laplacian $\Ll$ and a positive function $\phi$ on $[0,\Lambda]$. \\[1mm]  
\STATE \textbf{Calculate} $\phi(\mathbf{H}_m) \mathrm{F}_1$ and the basis $[\mathrm{Q}_1, \ldots, \mathrm{Q}_m]$ of the classical block Krylov space $\krycl_m(\Ll,\Ee_W)$ by using Algorithm \ref{algorithm-block-Lanczos}. \\[1mm] 
\STATE 
\textbf{Solve} the linear system of equations 
\begin{equation*}  
\Big( \mathrm{F}_1^* \phi(\mathbf{H}_m) \mathrm{F}_1 + \gamma N \mathbf{I}_N\Big) 
\underbrace{\begin{bmatrix} c_1^{(\mathrm{cbl})} \\ c_2^{(\mathrm{cbl})} \\ \vdots \\ c_N^{(\mathrm{cbl})} \end{bmatrix}}_{\mathrm{c}^{(\mathrm{cbl})}}
= \begin{bmatrix} y_1 \\ y_2 \\ \vdots \\ y_N \end{bmatrix}.
\end{equation*}
\STATE
\textbf{Calculate} the RLS predictor $y^{(\mathrm{cbl})}$ as
\[ y^{(\mathrm{cbl})} = [\mathrm{Q}_1, \ldots, \mathrm{Q}_m] \phi(\mathbf{H}_m) \mathrm{F}_1 
\mathrm{c}^{(\mathrm{cbl})} .\]
\
\end{algorithmic}

\end{algorithm}

\section{Computational complexity and storage requirements} \label{sec:costs}

The five introduced block Krylov subspace methods have different costs in terms of computational complexity and storage requirements. Depending on the problem at hand, these differences can get relevant in practical applications. For this, we provide a brief comparison of these costs for the block Krylov methods $\mathrm{kr} \in \{\mathrm{cbl}, \mathrm{gbl}, \mathrm{sbl}, \mathrm{cheb}\}$. The squared Chebyshev method $\mathrm{cheb}^2$ can be considered as a restarted variant of the Chebyshev method and the respective costs of $\mathrm{cheb}^2$ are therefore, up to a constant factor, the same as for the Chebyshev approximation.  
   
In view of the computational expanses of the methods, we compare in Table \ref{table:computationalcomplexity}, similarly as proposed in \cite{Gutknecht2006}, the following operations: the required matrix-vector products (MVs) in $\Rr^n$; the inner products in $\Rr^n$ (DOTs); the necessary vector space operations (addition and multiplication with scalars) in $\Rr^n$ (AXPYs); the calculation of the matrix function $\phi(\mathbf{H}_m)$ for the three Lanczos schemes, and the calculation of the coefficients $c_k(\phi)$ in case of the Chebyshev method. We assume that the calculation of the matrix function $\phi(\mathbf{H}_m)$ is performed by a direct algorithm using the spectral decomposition of the symmetric matrix $\mathbf{H}_m$. For this, we will also take into account that $\mathbf{H}_m$ has a sparse banded structure.  

The most expensive operations in Table \ref{table:computationalcomplexity} are, for large graph sizes $n$, the MVs. For a fixed iteration number $m$, the number of MVs is the same for all four methods. In the last section of this article, we will experimentally see that the classical block Lanczos method achieves a higher accuracy for the same number $m$ of iterations. Although this higher accuracy favors the usage of classical block Lanczos methods, the other methods are much cheaper in terms of required DOTs, AXPYs and the calculation of the matrix function $\phi(\mathbf{H}_m)$. In particular, if for sparse matrices $\Ll$ the cost of the MVs is not too dominant, the classical block Lanczos method can be outperformed by the other methods in terms of computational complexity if the block size $N$ gets large. 

\begin{table}
\begin{tabular}{|l|c|c|c|c|}
\hline Operations & \textrm{cbl} & \textrm{gbl} & \textrm{sbl}& \textrm{cheb}\\
\hline MVs & $m N$ & $m N$ & $m N$ & $m N$\\
DOTs & $\mathcal{O}(m N^2)$ & $\mathcal{O}(m N)$ & $\mathcal{O}(m N)$ & - \\
AXPYs & $\mathcal{O}(m N^2)$ & $\mathcal{O}(m N)$ & $\mathcal{O}(m N)$ & $\mathcal{O}(m N)$\\
$\phi(\mathbf{H}_m) / c_k(\phi)$ & $\mathcal{O}(m N^3) + \mathcal{O}(m^2 N^2)$ & $\mathcal{O}(m^2)$ & $\mathcal{O}(m^2 N)$ & $\mathcal{O}(m \log m)$ \\
\hline
\end{tabular}

\vspace{2mm}

\caption{Required operations to calculate $p_{\phi,m-1}^{(\mathrm{kr})}(\Ll)\Ee_W$ for the Krylov space methods $\mathrm{kr} \in \{\mathrm{cbl},\mathrm{gbl},\mathrm{sbl},\mathrm{cheb}\}$.}
\label{table:computationalcomplexity}
\end{table}

The competitivity of the classical block Lanczos method gets further diminished by the storage requirements during the computational process. As listed in Table \ref{table:memory}, this high memory demand is due to the fact that the entire Krylov basis $[\mathrm{Q}_1, \ldots, \mathrm{Q}_m]$ has to be stored in order to calculate the final matrix function. Comparing the three Lanczos schemes, the sequential block Lanczos scheme performs best in terms of storage requirements. It has the advantage that the operations on the columns of the block $\Ee_W$ can be performed independently, so that only the Krylov basis for one column is required at a particular point in time of the calculation. Particularly cheap in terms of memory requirements is however the Chebyshev method. As the three-term recurrence relation of the Chebyshev polynomials requires only $T_k(\mathbf{I}_n - \ts \frac{2}{\Lambda}\Ll) e_{\node{w}}$ and $T_{k-1}(\mathbf{I}_n - \ts \frac{2}{\Lambda}\Ll)e_{\node{w}}$ to calculate the next iterate $T_{k+1}(\mathbf{I}_n - \ts \frac{2}{\Lambda}\Ll) e_{\node{w}}$, only $2$ vectors have to be stored at a time for the calculation of the next basis vector. Also, the Chebyshev method can be applied sequentially so that only single vectors have to be stored and not the entire block.  

\begin{table}
\begin{tabular}{|l|c|c|c|c|}
\hline Storage & \textrm{cbl} & \textrm{gbl} & \textrm{sbl}& \textrm{cheb}\\
\hline $\mathrm{Q}_k/ T_k(\Ll) \Ee_W$ & $m n N$ & $m n N$ & $m n$ & $2n$ \\
$\mathbf{H}_m / c_k(\phi)$ & $\mathcal{O}(m N^2)$ & $\mathcal{O}(m)$ & $\mathcal{O}(m)$ & $m$ \\
\hline
\end{tabular}

\vspace{2mm}

\caption{Memory requirements for the calculation of the matrix polynomial $p_{\phi,m-1}^{(\mathrm{kr})}(\Ll)\Ee_W$ for the Krylov space methods $\mathrm{kr} \in \{\mathrm{cbl},\mathrm{gbl},\mathrm{sbl},\mathrm{cheb}\}$.}
\label{table:memory}

\end{table}

\begin{figure}[htbp]
	\centering Lanczos approximation \\[1mm]
	\includegraphics[width= \textwidth]{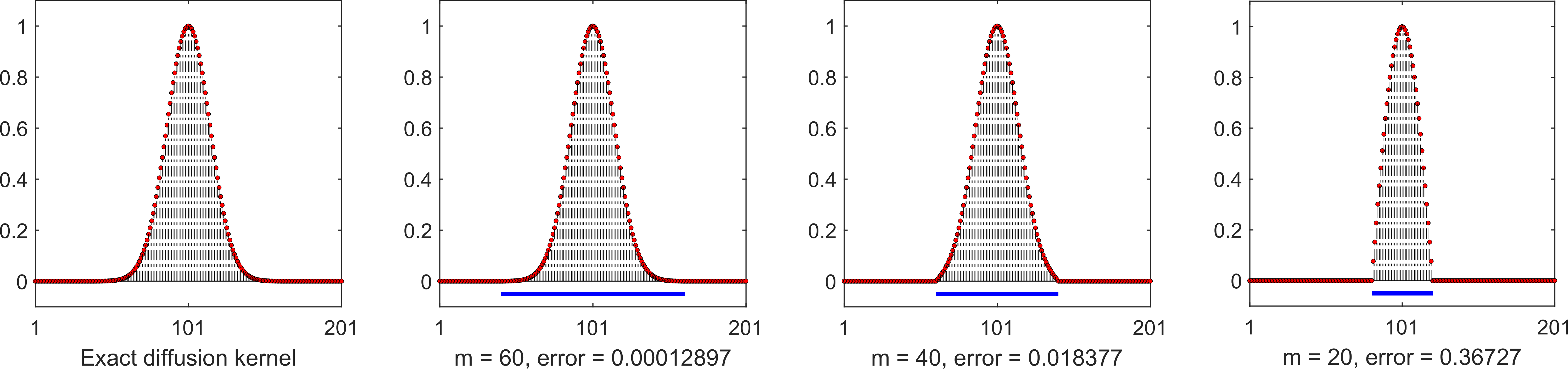} \\[1mm]
    \includegraphics[width= \textwidth]{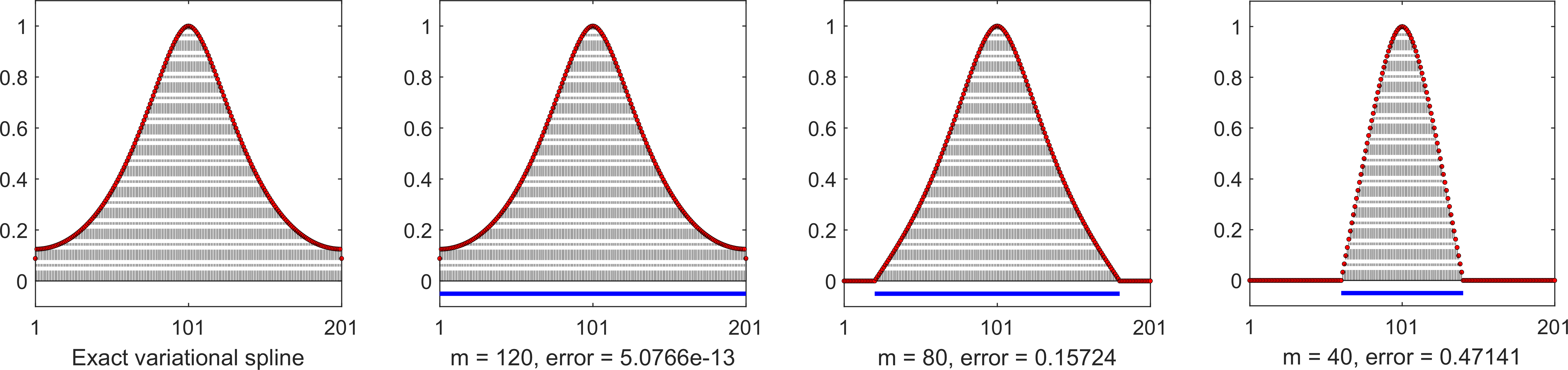} \\[2mm]
    Chebyshev approximation \\[1mm]
	\includegraphics[width= \textwidth]{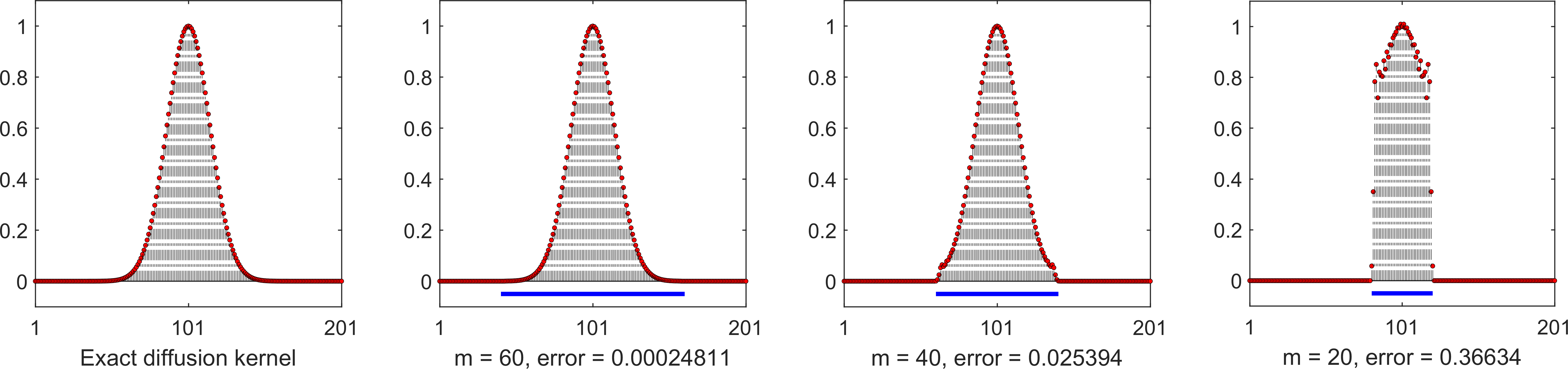} \\[1mm]
	\includegraphics[width= \textwidth]{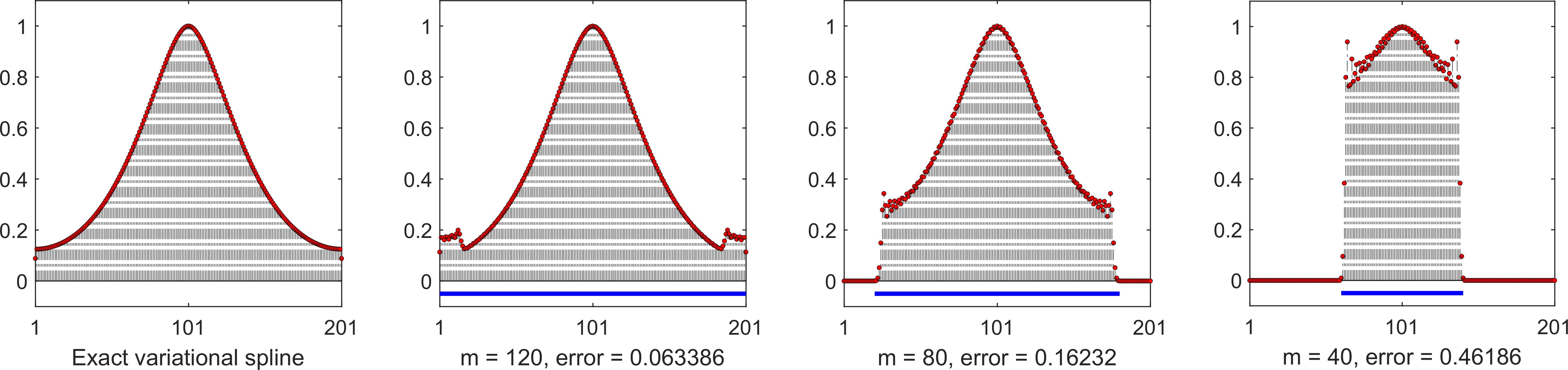} 
	\caption{Comparison between Lanczos and Chebyshev Krylov subspace approximation of a diffusion and a variation spline kernel on the path graph $G_1$. The blue line indicates the support of the approximant, the error with respect to the exact kernel column is measured in the uniform norm. }
	\label{fig:pathkernel}
\end{figure}

\section{Experiments} \label{sec:experiments}

For a concrete comparison of the five block Krylov methods for the approximation of the kernels and the respective kernel predictors, we conduct a series of experiments on a few simple data sets. The code used to conduct these experiments can be found in the freely available Github repository 

\begin{center}\url{https://github.com/WolfgangErb/GBFlearn}.\end{center} 

\subsection{Comparison between Lanczos and Chebyshev methods}

To visualize the differences in the convergence between the Lanczos method and the Chebyshev method for the approximation of $\phi(\Ll) \Ee_{W}$, we start with a simple path graph $G_1$ consisting of $201$ nodes and $200$ edges such that each edge connects two consecutive nodes. For simplicity, we start with only one central sampling node $W = \{\node{v}_{101}\}$. As kernels on $G_1$ we consider the diffusion kernel $\phi(\Ll) = e^{- t \Ll}$ with $t = 200$ as well as the variational spline kernel $\phi(\Ll) = (\Ll + \varepsilon)^{-s}$ with $\epsilon = 0.001$ and $s = 2$. In this example, $\Ll$ denotes the normalized graph Laplacian on $G_1$. Goal of our test is to see how fast a Lanczos method converges towards $\phi(\Ll) e_{\node{v}_{101}}$ compared to the approximation given by the Chebyshev polynomial. As the block size $N$ is equal to $1$, all Lanczos methods $\mathrm{kr} \in \{\mathrm{cbl}, \mathrm{gbl}, \mathrm{sbl}\}$ will provide the same approximant. The comparison to the Chebyshev approximation for different iteration numbers $m$ is visualized in Fig. \ref{fig:pathkernel}. It is visible that for fixed numbers $m$ the more adaptive Lanczos method provides smaller uniform errors for the kernel column $\phi(\Ll) e_{\node{v}_{101}}$ than the Chebyshev method.

\subsection{Convergence of block Krylov methods}

\begin{figure}[htbp]
	\centering
	\includegraphics[width= 0.6\textwidth]{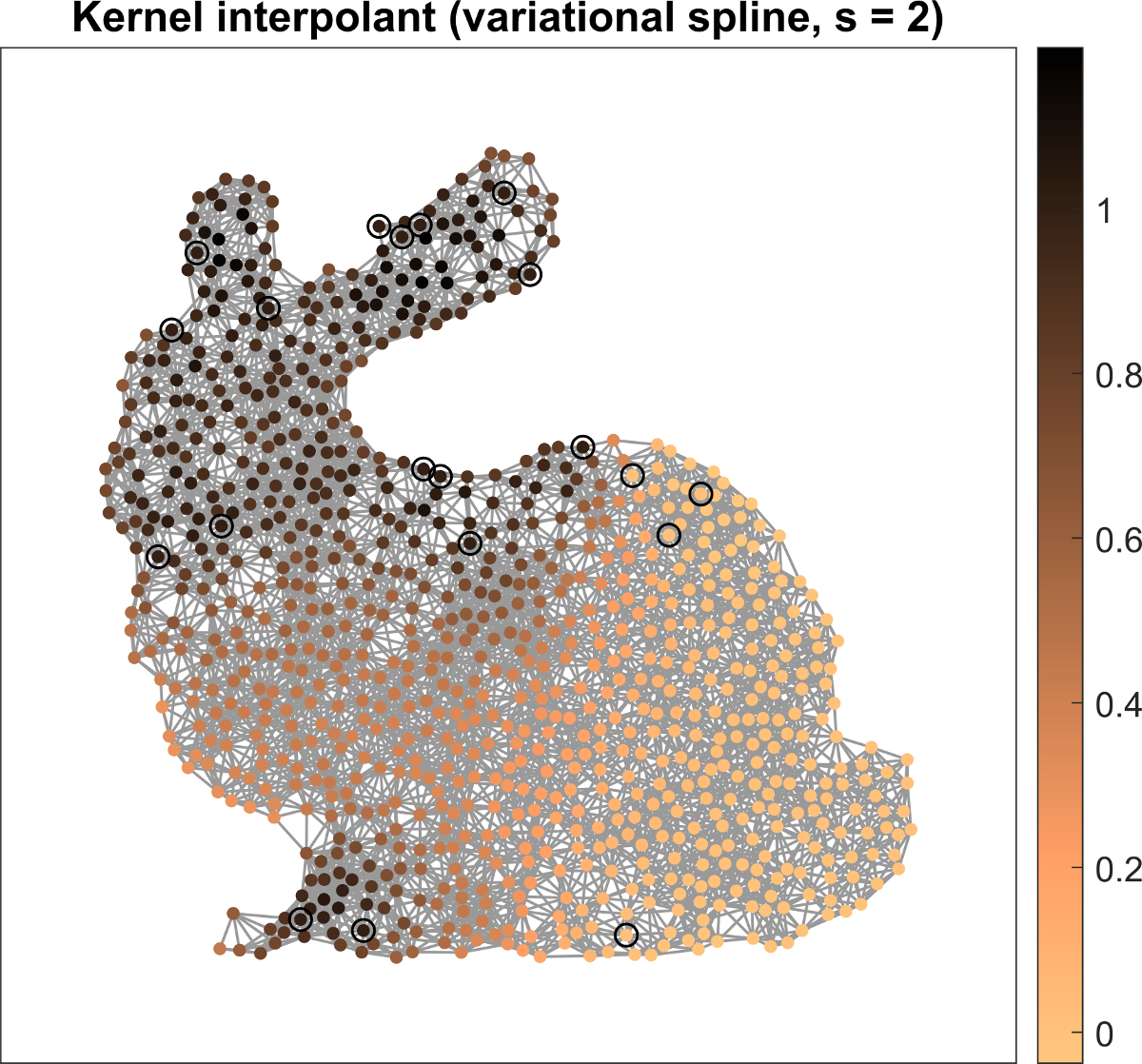}
	\caption{Kernel interpolant on the bunny graph using $N = 20$ samples and the variational spline kernel with parameters $s = 2$ and $\epsilon = 0.05$.}
	\label{fig:bunnykernel}
\end{figure}

In this section, we numerically evaluate how fast the approximate predictors $y^{(\mathrm{kr})}$ converge towards the kernel predictor $y$ for the five considered block Krylov methods $\mathrm{kr} \in \{\mathrm{cbl}, \mathrm{gbl}, \mathrm{sbl}, \mathrm{cheb}, \mathrm{cheb}^2\}$. This is relevant from a practical point of view since the iteration number $m$ of the block Krylov methods determines how many matrix-vector products are necessary during the calculations. As discussed in Section \ref{sec:costs}, the MVs are typically the most expensive ones if the size $n$ of the graph gets large.  

To compare the rate of convergence for the different block Krylov schemes, we use as a second test graph $G_2$ a reduced 2D projection of the Stanford bunny data set (Source: Stanford University Computer Graphics Laboratory). This data set consists of $n = 900$ points in the plane. To generate the graph, we connect two points of this set with an edge if the Euclidean distance between the points is smaller than a given radius $0.01$. This generates the graph $G_2$ with a total number of $7325$ unweigthed edges. As graph Laplacian $\Ll$ on $G_2$, we use the normalized graph Laplacian. As kernels we use again the diffusion kernel 
$\phi(\Ll) = e^{- t \Ll}$ with $t = 20$ and the variational spline kernel $\phi(\Ll) = (\Ll + \varepsilon)^{-s}$ with $\epsilon = 0.05$ and $s = 2$. We then calculate the kernel interpolants $y$ (this corresponds to a kernel predictor with $\gamma = 0$) based on $N = 20$ selected sampling nodes and binary labels $y_1, \ldots, y_N \in \{0,1\}$. The size of the graph is still moderate, so it is possible to calculate the kernel interpolant exactly without the usage of iterative methods. The exact interpolant based on the variational spline as kernel is plotted in Fig. \ref{fig:bunnykernel}.

\begin{figure}[htbp]
	\centering
	\includegraphics[width= 0.48\textwidth]{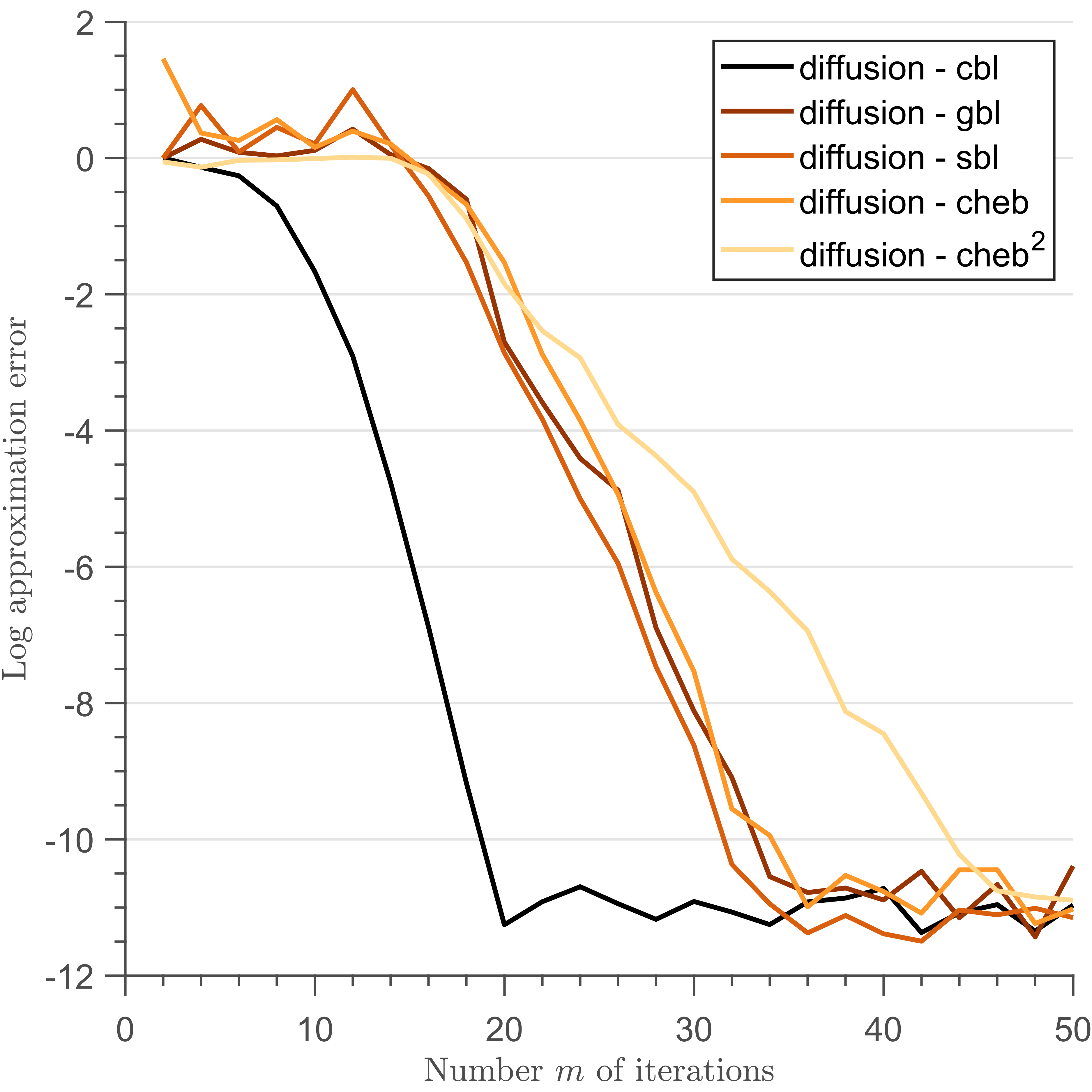} \quad
	\includegraphics[width= 0.48\textwidth]{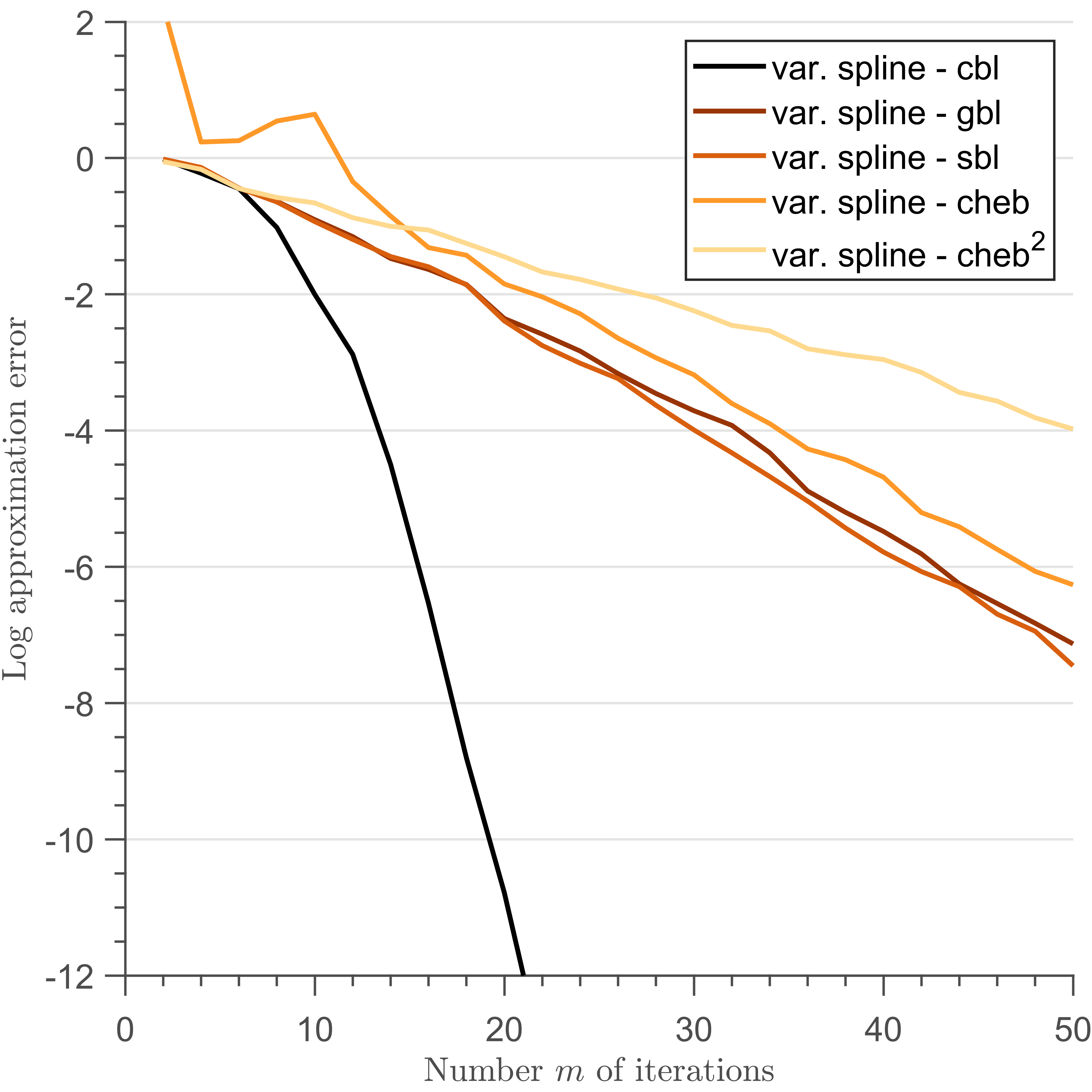}
	\caption{Uniform error $\|y - y^{(\mathrm{kr})}\|_\infty$ for the five block Krylov methods $\mathrm{kr} \in \{\mathrm{cbl}, \mathrm{gbl}, \mathrm{sbl}, \mathrm{cheb}, \mathrm{cheb}^2\}$ in terms of the iteration numbers $m$.}
	\label{fig:bunnyconvergence}
\end{figure}

The uniform error $\|y - y^{(\mathrm{kr})}\|_{\infty}$ for the approximation of the predictor $y$ with the five block Krylov methods $\mathrm{kr} \in \{\mathrm{cbl}, \mathrm{gbl}, \mathrm{sbl}, \mathrm{cheb}, \mathrm{cheb}^2\}$ is plotted in Fig. \ref{fig:bunnyconvergence}. It is visible that the classical block Lanczos method requires considerably less iterations $m$ for convergence compared to the other methods, followed by the global and the sequential block Lanczos method. The two Chebyshev methods display a slightly slower convergence in terms of the number $m$ of iterations. While this seems to be an indication to use the classical block Lanczos scheme in practice, we have already seen in Section \ref{sec:costs} that this method has a considerably higher memory demand and it requires larger computational times for the algorithmic operations aside the MVs. For this, the usage of the classical block Lanczos method can only be recommended in those cases in which the MVs form the dominant part of the computational expanses. 

\begin{figure}[htbp]
	\centering
	\includegraphics[width= 0.6\textwidth]{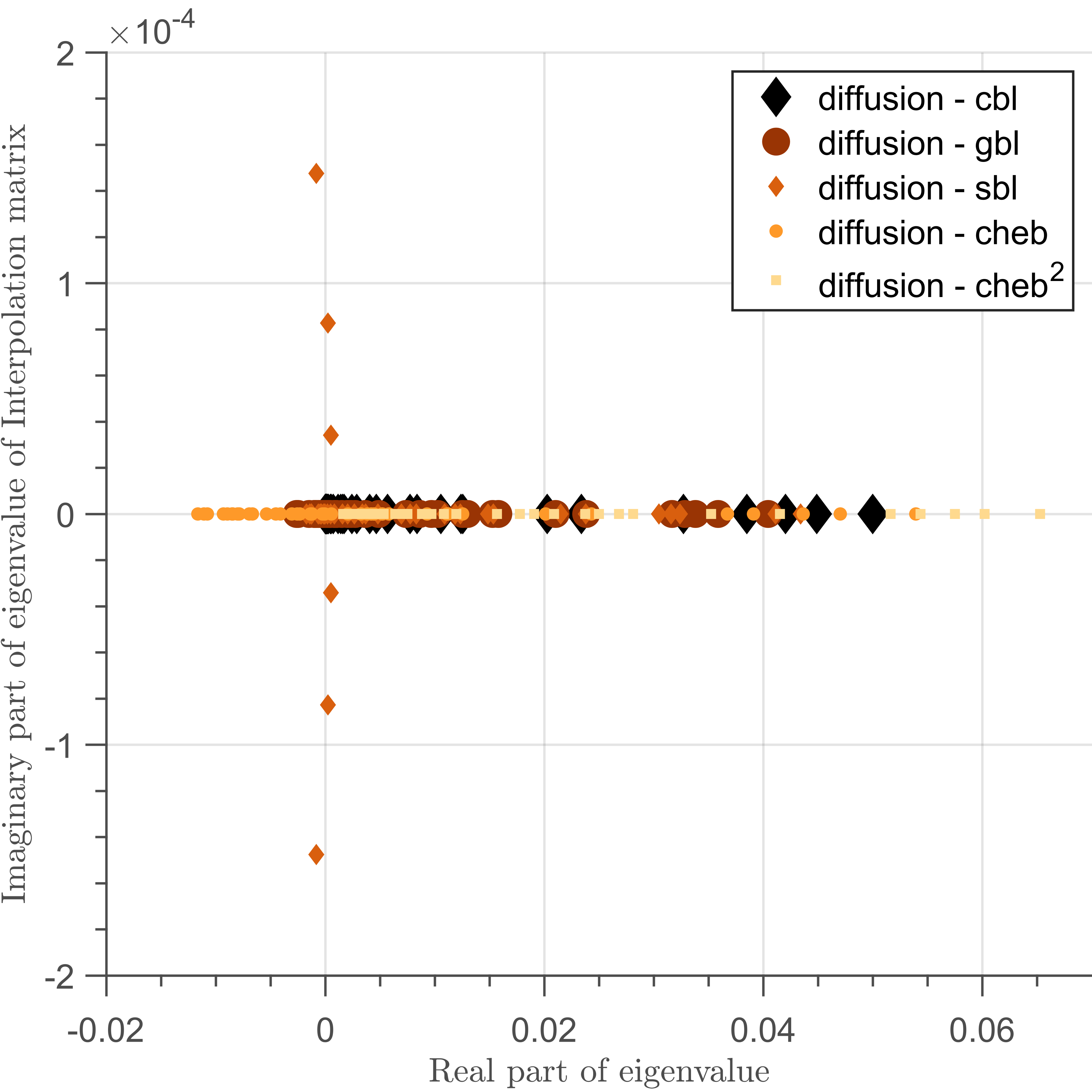}
	\caption{Eigenvalues of $\Ee_W^* p_{\phi,5}^{(\mathrm{kr})}(\Ll) \Ee_W$ for the five methods $\mathrm{kr} \in \{\mathrm{cbl}, \mathrm{gbl}, \mathrm{sbl}, \mathrm{cheb}, \mathrm{cheb}^2\}$, with kernel $\phi(\Ll) = e^{-t \Ll}$, $t = 20$. }
	\label{fig:bunnyeigenvalues}
\end{figure}

\subsection{Eigenvalues of the collocation matrices}

We have seen in Theorem \ref{thm:classicalpositivedefinite} and Corollary \ref{cor:1} that the collocation matrices $\Ee_W^* p_{\phi,m-1}^{(\mathrm{cbl})}(\Ll) \Ee_W$ used for the computation of the predictor $y^{(\mathrm{cbl})}$ are positive definite and that the uniqueness of $y^{(\mathrm{cbl})}$ is always guaranteed. Also for the squared Chebyshev polynomials the positive semi-definiteness of the matrix $\Ee_W^* p_{\phi,m-1}^{(\mathrm{cheb}^2)}(\Ll) \Ee_W$ is ensured. For the other three block Krylov methods $\mathrm{kr} \in \{\mathrm{gbl}, \mathrm{sbl}, \mathrm{cheb}\}$ the positive definiteness is in general not given. We illustrate this with a numerical counterexample. For the bunny graph $G_2$ introduced in the previous section, we pick $N = 40$ sampling nodes, the diffusion kernel $\phi(\Ll) = e^{-t \Ll}$, $t = 20$, and calculate the eigenvalues of the matrices $\Ee_W^* p_{\phi,m-1}^{(\mathrm{kr})}(\Ll) \Ee_W$, selecting $m = 6$. These eigenvalues are plotted in Fig. \ref{fig:bunnyeigenvalues}. Corresponding to the theoretic results, the eigenvalues for the methods $\mathrm{kr} \in \{\mathrm{cbl}, \mathrm{cheb}^2\}$ turn out to be positive. On the other hand,  for $\mathrm{kr} \in \{\mathrm{gbl}, \mathrm{cheb}\}$ we get also negative eigenvalues and for $\mathrm{kr} = \mathrm{sbl}$ even complex-valued numbers. This implies that for small iteration numbers $m$ the matrix $\Ee_W^* p_{\phi,m-1}^{(\mathrm{kr})}(\Ll) \Ee_W + N \gamma \mathbf{I}_N$ is not necessarily invertible in case of the Krylov methods $\mathrm{kr} \in \{\mathrm{gbl}, \mathrm{sbl}, \mathrm{cheb}\}$. On the other hand, for increasing $m$ the matrices $\Ee_W^* p_{\phi,m-1}^{(\mathrm{kr})}(\Ll) \Ee_W$ converge towards $\Ee_W^* \phi(\Ll) \Ee_W$. This makes sure that for large enough $m$ the uniqueness of the predictor $y^{(\mathrm{kr})}$ is given for all five block Krylov methods as soon as $\phi$ is positive on $[0,\Lambda]$.   

\section{Conclusion}
In this article, we have investigated and compared five block Krylov subspace methods for the iterative calculation of kernel matrices and kernel predictors on graphs: three Lanczos-type methods and two Chebyshev methods. From a theoretical point of view the classical block Lanczos method has some important advantages compared to the other four block Krylov methods: it guarantees the uniqueness of the kernel predictor, the calculation of the predictor can be performed without the explicit knowledge of the kernel matrix and it displays a much faster convergence in terms of the number of matrix-vector products. On the other hand, the classical block Lanczos method has a considerably larger cost in terms of memory and a larger computational complexity beyond the matrix-vector products. For this, in practical calculations the global and the sequential block Lanczos methods can outperform the classical block Lanczos iteration if the costs of the matrix-vector products are not too dominant. Although the two considered Chebyshev methods usually require more iterations for convergence they are valuable alternatives to the block Lanczos methods in case of limited memory.    

\section*{Acknowledgment}

The author acknowledges support by GNCS-IN$\delta$AM, the Rete ITaliana di Approssimazione (RITA) and the thematic group on Approximation Theory and Applications of the Italian Mathematical Union.

\end{document}